\documentclass[12pt,english,a4paper]{smfart}

\usepackage[T1]{fontenc}
\usepackage{lmodern}
\usepackage{smfthm, dsfont}
\usepackage[headings]{fullpage}
\usepackage{amssymb, mathrsfs, hyperref,stmaryrd, color, comment}
\usepackage[dvipsnames]{xcolor}
\usepackage{tikz}
\usepackage{graphicx} 

\usetikzlibrary {positioning}

\renewcommand{\geq}{\geqslant}
\renewcommand{\leq}{\leqslant} 
\renewcommand{\ge}{\geqslant}
\renewcommand{\le}{\leqslant} 

\newcommand{\R}{\mathbb{R}}
\newcommand{\Z}{\mathbb{Z}}
\newcommand{\Vis}{\mathrm{Vis}}

\newcommand{\defini}{\textbf}

\definecolor{col}{RGB}{180, 15, 20}

\NumberTheoremsIn{section}

\date{September 2025}

\thanks{The authors wish to thank Marcelo Hil\'ario for bibliographical help.}

\author{Samuel Le Fourn}
\address{Université Grenoble Alpes, Institut Fourier \\ 100 rue des mathématiques, 38610 Gières, France}
\email{samuel.le-fourn@univ-grenoble-alpes.fr}
\urladdr{\href{https://www-fourier.univ-grenoble-alpes.fr/~lefourns}{\url{www-fourier.univ-grenoble-alpes.fr/~lefourns}}}

\author{Mike Liu}
\address{ENSAE, Fairplay joint team, CREST France\\
5 avenue Le Chatelier\\
91120 Palaiseau, France}
\email{mike.liu@ensae.fr}

\author{Sébastien Martineau}
\address{Sorbonne Universit\'e, LPSM\\
 4 place Jussieu, Case 158\\
 75252 Paris Cedex 05, France}
\email{smartineau@lpsm.paris}
\urladdr{\href{https://martineau-maths.github.io/webpage/}{martineau-maths.github.io/webpage}}

\title{Percolative properties of the random coprime colouring}

\begin{abstract}
   Given $u$ and $v$ in $\mathbb{Z}^d$, say that $u$ is visible from $v$ if the segment from $u$ to $v$ contains exactly two elements, which are $u$ and $v$. Take $X$ ``uniformly at random in $\mathbb{Z}^d$'' and colour each vertex $u$ of $\mathbb{Z}^d$ in white if $u$ is visible from $X$ and in black otherwise. Previous independent works of Pleasants--Huck and of the third author give a precise meaning to this definition.
   This paper is dedicated to the study of this random colouring from the point of view of percolation theory: given a reasonable graph structure on $\mathbb{Z}^d$, how many infinite black (resp. white) connected components are there?{\color{blue}}
\end{abstract}

\begin{document}

\maketitle

\section{Introduction}

Let $d\ge1$ and let $\Gamma$ be a lattice in $\R^d$, i.e. a discrete subgroup of $\R^d$ such that $\R^d/\Gamma$ is compact, for example $\Gamma=\Z^d$. Given some $u\in \Gamma$, declare an element $v$ of $\Gamma$ to be \defini{visible from $u$ (in $\Gamma$)} if $v$ is distinct from $u$ and the line segment $[u,v]$ intersects $\Gamma$ only at $u$ and $v$. We denote by $\Vis_\Gamma(u)$ the set of all vertices that are visible from $u$ in $\Gamma$. The lattice $\Gamma$ and the origin $u$ being fixed, can we get some useful description of $\Vis_\Gamma(u)$?

In some sense, it is easy to answer this question. 
Up to translating, we may assume that $u$ is equal to $0$. In that case, one easily checks that $\Vis_{\Gamma}(0)=\Gamma\setminus\bigcup_p p\Gamma$, where we use the following convention. 

\begin{enonce*}[remark]{Convention}\label{convention}
Throughout this paper, when $p$ appears as a subscript of an operator such as $\sum$, $\prod$ or $\bigcup$, it is implicitly meant that $p$ ranges over the set of prime numbers. For instance, if we write $\sum_{p>30}$, this means that we sum over primes $p$ that furthermore satisfy the property $p>30$.
\end{enonce*}

What does happen if we pick a random element $X$ in $\Gamma$ and try to understand the distribution of the set $\Vis_\Gamma(X)$ of all elements visible from $X$ in $\Gamma$ ? This depends on the distribution of $X$. The most natural choice would be to pick $X$ ``uniformly at random'' in $\Gamma$. As such, it makes no sense, as there is no uniform probability measure on the infinite discrete group $\Gamma$ --- Haar measures on $\Gamma$ are multiples of the counting measure, and none of these multiples has total mass 1. But, even though ``picking $X$ uniformly at random in $\Gamma$'' makes no sense, it turns out that it is possible to make sense of ``the distribution of $\Vis_\Gamma(X)$ for $X$ picked uniformly at random in $\Gamma$''. This was done in \cite{pleasantshuck, martineau}. Let us first explain what is this distribution. We defer to Section~\ref{sec:convergence} explanations regarding in which sense this distribution indeed deserves to be considered as ``the distribution of $\Vis_\Gamma(X)$ for $X$ picked uniformly at random in $\Gamma$''.

For each prime $p$, pick one of the $p^d$ cosets (a.k.a. translates) of $p\Gamma$ in $\Gamma$ uniformly at random, and do so independently for all primes. Write ${B}_p$ the coset selected for the prime $p$. Let us consider 
\[
{W}:=\Gamma\setminus \bigcup_p {B}_p,
\]
which is a random variable taking values in $\mathscr{P}(\Gamma)$. The set $\mathscr{P}(\Gamma)$ of all subsets of $\Gamma$ is endowed with the smallest $\sigma$-field such that for every $v\in \Gamma$, the map $P\longmapsto \mathds{1}_P(v)$ is measurable. This is also the smallest $\sigma$-field such that for every finite subset $F$ of $\Gamma$, the map $P\longmapsto P\cap F$ is measurable, its codomain being $\mathscr{P}(F)$ endowed with the $\sigma$-field of all subsets of  $\mathscr{P}(F)$. The probability measure $\rho$ we shall be interested in is the probability distribution of ${W}$, which is a probability measure on $\mathscr{P}(\Gamma)$. In other words, for every measurable $\mathcal{A}\subset \mathscr{P}(\Gamma)$, we set $\rho(\mathcal{A}):=\mathbf{P}({W}\in \mathcal{A})$.

\begin{figure}[h!!]
\centering
    \vspace{0.2cm}
    \includegraphics[width=9.1cm, height=9.1cm]{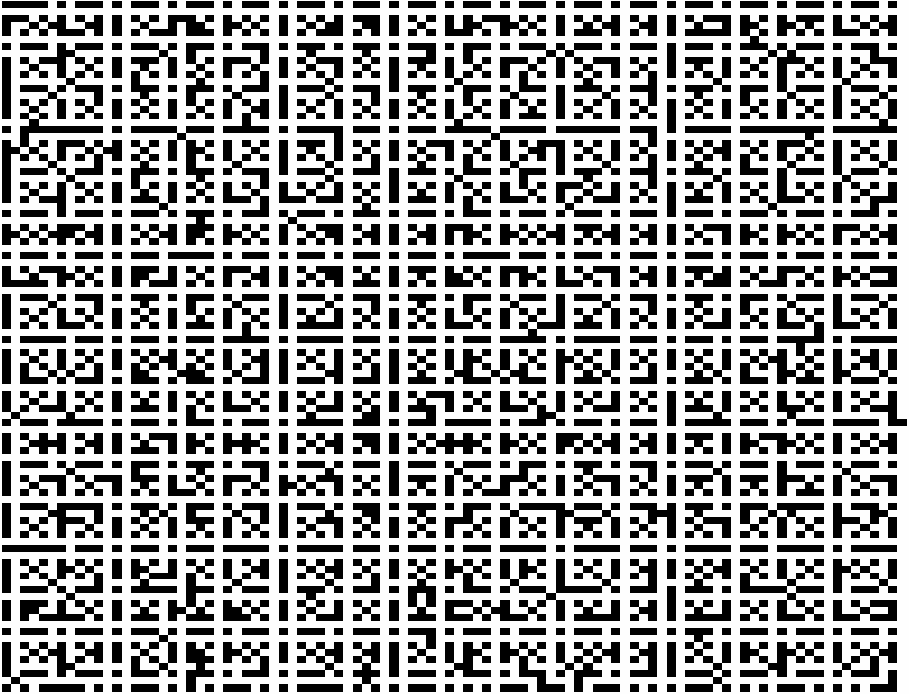}
    \vspace{0.2cm}
    \caption{A picture of $W$ for $\Gamma=\mathbb{Z}^2$. Elements of $W$ are depicted by white squares and elements of its complement are depicted by black squares. Which square corresponds to $(0,0)$ is irrelevant as the probability distribution of $W$ is invariant under translations.}
    \label{fig:simu}
\end{figure}

This paper is dedicated to studying $W$ and its complement from the point of view of percolation theory. Say that some $v\in\Gamma$ is \defini{white} if it belongs to $W$ and \defini{black} if it does not. Our task is to count how many infinite white (resp. black) connected components there are in this random colouring --- which can be called the \defini{random coprime colouring}, see Section~\ref{sec:primi}. However, for this question to make sense, we need to put a graph structure on $\Gamma$. Let us start with the most natural candidates, namely $\Gamma=\mathbb{Z}^d$ and the graph structure being the square lattice if $d=2$, the cubic lattice if $d=3$, etc.

\begin{figure}[h!!]
\centering
    \includegraphics[width=10cm]{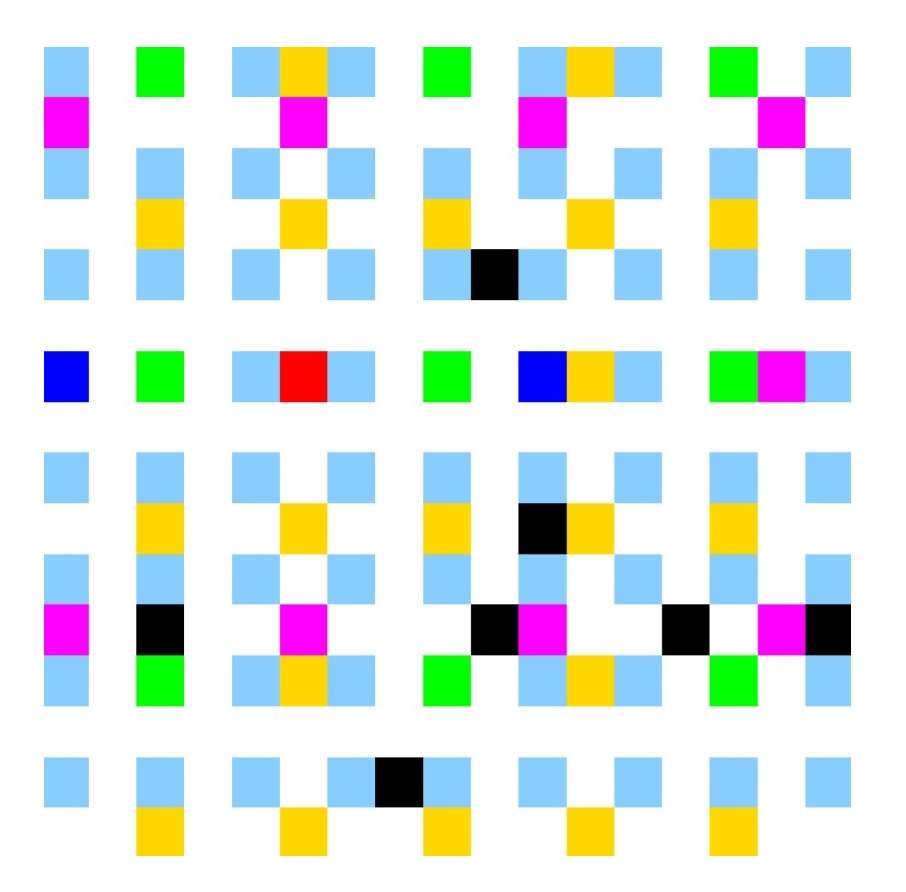}
    \caption{A picture of $W$ and its complement for $\Gamma=\mathbb{Z}^2$. Elements of $W$ are depicted by white squares. Elements of its complement, namely $\bigcup_p B_p$, are either in black or in colour. Belonging to $B_p$ for $p=2$ (resp. 3, 5) is indicated by the colour cyan (resp. yellow, magenta). Belonging to several $B_p$ for primes $p\le 5$ is represented by mixing additively the corresponding colours. It turns out that, in this picture, there is no conflict regarding the use of the colour white: no square of the depicted portion of $\mathbb{Z}^2$ belongs to $B_2\cap B_3\cap B_5$, so that white is never obtained by additive mixing to represent one of these squares that does not belong to $W$.}
    \label{fig:simu-bis}
\end{figure}

Adhering to the terminology of percolation theory, a graph structure on $\Gamma$ being fixed, we call \defini{black cluster} (resp. \defini{white cluster}) a connected component of black vertices (resp. white vertices).

\begin{theo}\label{thm:usual}
    Let $d\ge 2$ and let $\Gamma=\mathbb{Z}^d$ be endowed with its usual graph structure, meaning that two vertices are adjacent if and only if they differ at exactly one coordinate and by exactly 1, in absolute value.

    Then, the number of infinite white clusters is almost surely equal to 1, and that of infinite black clusters is almost surely equal to 0.
\end{theo}

\begin{rema}
The case $d=1$ is simple and different, as $W$ is almost surely empty. Indeed, the probability that a given element of $\mathbb{Z}$ belongs to $W$ is $\prod_p(1-\frac 1 p)$, which is equal to 0 as $\sum_p \frac1p=\infty$. Recall the convention of page~\pageref{convention}: we do not use the divergence of the harmonic series but the divergence of the sum of inverse of \emph{primes}.
\end{rema}

\begin{rema}
    Theorem~\ref{thm:usual} was derived in \cite[end of Section~2]{martineau} but the proof was very indirect. It had to combine the main theorem of \cite{martineau} with the work of \cite{vardi}, which itself built upon earlier works of Friedlander \cite{friedlander}. The present paper gives a direct proof of Theorem~\ref{thm:usual} which is much shorter, simpler, and more elementary.
\end{rema}

What can be said beyond the specific graphs of Theorem~\ref{thm:usual}? Let us ask the question in general terms and answer what we can. This will in particular shed some light on the key properties of these graphs that are used in the proof.
A natural and general setup to endow $\Gamma$ with a good graph structure consists in picking a Cayley graph structure on $\Gamma$. Let us recall some more definitions regarding Cayley graphs.

In our graphs, edges are unoriented, there are no multiple edges and no self-loops. In other words, we work with \emph{simple graphs}: a graph structure on a set is simply an irreflexive symmetric relation on it. Given a graph structure on $\Gamma$, the following conditions are equivalent \cite{sabidussi}:
\begin{enumerate}
    \item\label{item:conceptual} the graph is connected, every vertex has finitely many neighbours, and for every $a\in \Gamma$, the map $u\longmapsto u+a$ is a graph-automorphism,
    \item\label{item:concrete} there is a finite subset $S\subset \Gamma$ that generates $\Gamma$ as a group such that for every $(u,v)\in\Gamma^2$, the vertices $u$ and $v$ are adjacent if and only if $u-v\in S$.
\end{enumerate}
When these conditions hold, we say that we have a \defini{Cayley graph} of $\Gamma$. Notice that if we have a simple graph (the edges of which are unoriented and join distinct vertices), then any $S$ in Condition~\ref{item:concrete} will automatically satisfy $S=-S$ and $0\notin S$. Conversely, let us call a generating subset $S$ of $\Gamma$ \defini{admissible} if it is finite, and satisfies $S=-S$ and $0\notin S$. We then declare $u$ and $v$ in $\Gamma$ to be adjacent if and only if $u-v\in S$. Such a choice of $S$ therefore produces a Cayley graph structure on $\Gamma$, and it is the point of view we shall use here. We will not rely on its equivalence with the conceptual definition given by Condition~\ref{item:conceptual}, which we only mentioned because it stresses how natural the Cayley property is. The general problem we investigate goes as follows.

\begin{enonce}{Problem}
    \label{problem}
    Let $d\ge 2$ and let $\Gamma$ be a lattice in $\mathbb{R}^d$. For every admissible generating subset $S$ of $\Gamma$, answer the following questions, being understood that $\Gamma$ is endowed with the structure of Cayley graph that $S$ induces: how many infinite white  clusters are there? how many infinite black clusters are there?
\end{enonce}

Observe that the answer to the question may depend on $\Gamma$, $S$, and our choice of a colour, but not on the randomness of our configuration. Indeed, the probability measure $\rho$ is ergodic under the action of $\Gamma$ --- see e.g. \cite[page~9]{martineau}. Therefore, for all $\Gamma$, $S$, and colour, the answer to the question has an almost sure value.

We do not fully solve Problem~\ref{problem}. Still, we are able to go beyond Theorem~\ref{thm:usual}. As a warm-up, let us explain why certain choices of $S$ do lead to the existence of an infinite black cluster. As soon as $S$ contains some $u$ that can be written $u=nv$ for some $n\ge 2$ and $v\in \Gamma$, \label{page:exist-black}there is almost surely at least one infinite black cluster. Given such $(u,n,v)$, it suffices to pick a prime divisor $p$ of $n$ and to observe that $B_p$ is a black set containing a translate of the infinite path $(0,u,2u,3u,\dots)$. One can also produce examples of $S$ such that no element of $S$ can be written as $nv$ for $n\ge2$ and $v\in\Gamma$ but there is still at least one infinite black cluster. For instance, taking $\Gamma=\mathbb{Z}^d$, if $p$ and $q$ are distinct primes and if $S$ contains $v:=(p,-q,0,\dots,0)$ and $w:=(-p,pq+q,0,\dots,0)$, then the black set $B_p\cup B_q$ contains translates of the infinite path $(0\,,\,v\,,v+w\,,\,v+w+v\,,\,v+w+v+w\,,\,\dots)$. Indeed, the Chinese Remainder Theorem guarantees that $B_p\cup B_q$ is a translate of $p\mathbb{Z}^d\cup q\mathbb{Z}^d$, in which starting the path at $(0,q,0,\dots,0)$ works. You may also argue that $B_p\cup B_q$ has a positive probability to be $p\mathbb{Z}^d\cup q\mathbb{Z}^d$ and conclude by ergodicity.

Instead of going further in the direction of counterexamples, an interesting task is to try and extend the techniques used to establish Theorem~\ref{thm:usual} in order to prove that some natural examples of $(\Gamma,S)$ share (at least part of) the conclusion of Theorem~\ref{thm:usual}. In other words, we look for nice examples or setups where we can guarantee one or several of the following statements:
\begin{itemize}
    \item existence of an infinite white cluster,
    \item uniqueness of the infinite white cluster,
    \item inexistence of the infinite black cluster.
\end{itemize}
Let us point out that uniqueness of the infinite white cluster is a proper field of investigation here. Indeed, the random coprime colouring is not insertion-tolerant at all --- see Section~\ref{sec:context-perco} for definitions and a proof. Because of that, we cannot use the classical Burton--Keane argument \cite{burton-keane} to prove uniqueness of the infinite white cluster.

Regarding existence of an infinite white cluster, here is an easy proposition.

\begin{prop}[Existence of an infinite white cluster]
    \label{prop:exist-white}
    Let $d\ge \color{col}3$ and let $\Gamma$ be a lattice in $\mathbb{R}^d$. Let $S$ be an admissible generating subset of $\Gamma$. Endow $\Gamma$ 
 with the Cayley graph structure given by $S$.

    Then, there is almost surely at least one infinite white cluster.
\end{prop}

This proposition is proved as a warm-up in Section~\ref{sec:warm-up-general}. Therefore, as far as existence of the infinite white cluster is concerned, only the 2-dimensional case of Theorem~\ref{thm:usual} is substantial.

Regarding uniqueness of the infinite white cluster and inexistence of the infinite black cluster, all dimensions are interesting in Theorem~\ref{thm:usual} and the general case is deduced from the 2-dimensional one. For these properties, what can be done beyond Theorem~\ref{thm:usual}? Section~\ref{sec:gen} is mainly dedicated to answering this question. As some statements there are lengthy and better appreciated after reading Section~\ref{sec:reduction}, we do not include them in the introduction and defer their statements to Section~\ref{sec:gen-statements}. For now, we only gather some concrete consequences of them. In the next theorems, we refer to some interesting lattices by their name: these lattices are introduced in Section~\ref{sec:interestinglattices}.

\begin{theo}[Existence and uniqueness of the infinite white cluster]
    \label{thm:examples}
    Let $\Gamma\subset \mathbb{R}^d$ be any of the following lattices: the triangular lattice in $\mathbb{R}^2$, the $D_d$ lattice in $\mathbb{R}^d$ for $d\ge2$, the $E_8$ lattice in $\mathbb{R}^8$, or the Leech lattice in $\mathbb{R}^{24}$. Let $S$ be the set of elements of $\Gamma\setminus \{0\}$ with minimal norm in the Euclidean space $\mathbb{R}^d$, which indeed is an admissible generating subset of $\Gamma$. Endow $\Gamma$ with the Cayley graph structure given by $S$.

    Then, the number of infinite white clusters is almost surely equal to 1.
\end{theo}

\begin{theo}[Inexistence of the infinite black cluster]
    \label{thm:examples-black}
    Let $\Gamma\subset \mathbb{R}^d$ be either the triangular lattice in $\mathbb{R}^2$ or the $D_d$ lattice in $\mathbb{R}^d$ for $d\ge2$. Let $S$ be the set of elements of $\Gamma\setminus \{0\}$ with minimal norm in the Euclidean space $\mathbb{R}^d$, which indeed is an admissible generating subset of $\Gamma$. Endow $\Gamma$ with the Cayley graph structure given by $S$.

    Then, there is almost surely no infinite black cluster.
\end{theo}

\begin{theo}[Spread-out graphs]
    \label{thm:spread-out}
    Let $d\ge 2$ and let $\Gamma=\mathbb{Z}^d$. Let $p\in[1,\infty]$ and $\alpha\in [1,\infty)$. At last, let $S=\{x\in \mathbb{Z}^d\,:\,0<\|x\|_p\le \alpha\}$ and endow $\Gamma$ with the Cayley graph structure given by $S$.

    Then, the number of infinite white clusters is almost surely equal to 1.
\end{theo}

Observe that, under the assumptions of Theorem~\ref{thm:spread-out}, it may well be that there is at least one infinite black cluster --- by what we have seen on page~\pageref{page:exist-black}, this occurs as soon as $\alpha\ge 2$.


\vspace{0.27cm}

If you wish to learn more about context and motivations around this study, please have a look at the next section. If you want to dive directly into the proof, proceed directly to Section~\ref{sec:reduction}, which reduces Theorem~\ref{thm:usual} to proving an estimate on ``the probability to cross a rectangle''. Section~\ref{sec:goodcrossing}, which mainly lies in the realm of number theory, establishes such an estimate. Then, in Section~\ref{sec:gen}, we extend the scope of the techniques used in Section~\ref{sec:reduction} to cover other $(\Gamma,S)$ of interest.

\section{Motivations and context}
\label{sec:motivation}
Here is the context surrounding the problem studied in this paper.


\subsection{Visible lattice points have several guises}
\label{sec:primi}
Different mathematicians may encounter visible lattice points from different angles. This provides additional interest for the notion, as well as additional terminology that may interfere with bibliographical search.


An element $u$ of $\Gamma$ is said to be \emph{primitive} in $\Gamma$ if it does not belong to any $n \Gamma$ with $n \geq 2$. In particular, an element $u$ of $\Gamma$ is primitive if and only if it is visible from the origin. A primitive element $u$ of $\Gamma$ can always be completed to a $\Z$-basis $(u,v_1, \cdots, v_{d-1})$ of $\Gamma$. For $\Gamma=\mathbb{Z}^d$, this is also equivalent to the entries of $u$ being \emph{coprime}, meaning that $\gcd(u_1,\dots,u_d)=1$.

\subsection{In which sense is $\rho$ the correct object to consider?}\label{sec:convergence}
We now explain why $\rho$ deserves to be considered as ``the distribution of $\Vis_\Gamma(X)$ for $X$ picked uniformly at random in $\Gamma$''. This is for motivation only and is not to be used in this paper.

For every $n\ge 0$, let $F_n:=\{x\in\Gamma\,:\,\|x\|_2\le n\}$ and let $X_n$ be a random point chosen uniformly at random in $F_n$. Endow the set $\mathscr{P}(\Gamma)$ with the smallest topology making $P\longmapsto \mathds{1}_P(v)$ continuous for every $x\in \Gamma$. It has been proved that $\Vis_{X_n}(\Gamma)$ converges in distribution to $\rho$ when $n$ goes to infinity. This means that for every continuous (automatically bounded) map $f:\mathscr{P}(\Gamma)\to \mathbb{R}$, the expectation of $f(\Vis_{X_n}(\Gamma))$ converges to $\int f\,\mathrm{d}\rho$. Equivalently, it means that for every finite subset $F$ of $\Gamma$ and every $\mathcal{A}\subset \mathscr{P}(F)$, the probability $\mathbf{P}(\Vis_{X_n}(\Gamma)\in \mathcal{A})$ converges to $\rho(\{P\in\mathscr{P}(\Gamma):P\cap F\in \mathcal{A}\})$ when $n$ goes to infinity.

This result was established in its second form in \cite{pleasantshuck} when $\Gamma=\mathbb{Z}^d$. The idea there was mainly that every pattern $\mathcal{A}$ admits an asymptotic density. In \cite{martineau}, instead of thinking of one convergence of number for every choice of $\mathcal{A}$, the same problem was considered in the framework of convergence of probability measures. The lattice was again taken to be $\mathbb{Z}^d$ but rather general sequences $(F_n)$ were taken care of. We do not provide detailed statements here and refer the interested reader to \cite{martineau}. As the assumptions of \cite{martineau} on $(F_n)$ are stable under linear automorphisms of $\mathbb{R}^d$, the convergence readily adapts to any lattice $\Gamma$ for rather general sequences $(F_n)$, including that of the previous paragraph.

Let us also point out that $W$ is indeed constructed from a point chosen ``uniformly at random'' --- simply, it is not chosen in $\mathbb{Z}^d$ but elsewhere. Here, by ``uniformly at random'', we mean ``according to the Haar measure of some compact Hausdorff topological group''. There is no group topology making $\mathbb{Z}^d$ compact and Hausdorff but we can realise $\mathbb{Z}^d$ as a dense subgroup in some compact Hausdorff group, for instance $\prod_p (\Z/p\Z)^d$. And indeed, $W$ is directly constructed from $(B_p)$, which is Haar-distributed in $\prod_p (\Z/p\Z)^d$. More details are given in \cite{martineau}, where the compactification $\hat{\Z}=\varprojlim\Z/n\Z$ of $\mathbb Z$ is also used. See also \cite{kubotasugita}, where the authors use such compactifications to revisit the classical result according to which given $d$ independent random integers chosen ``uniformly at random in $\mathbb Z$'', the probability that they are coprime is $\frac{1}{\zeta(d)}$.

\subsection{Why is it relevant to study the percolative properties of the random coprime colouring?}\label{sec:context}
As explained in the previous paragraph, the random variable $W$ is a natural object of study. Many questions can be asked about it, including questions of percolation theory. Doing so is a source of new arithmetic problems, and reveals new phenomena in percolation theory. Indeed, percolation models that are usually studied enjoy some form of ``mixing'': knowing what happens in some region gives little or no information about what happens very far away. For this model, this is the opposite situation.

Let us see an instance of this phenomenon. Let $d\ge 2$ and let $W=\Gamma\setminus\bigcup_p B_p$. If you are given the restriction of $W$ to the half-space $\{x\in\mathbb{Z}^d\,:\,x_1\ge0\}$, then you can almost surely predict $W$ everywhere. Indeed, one easily checks that for every prime $p$, there is a unique coset of $p\mathbb{Z}^d$ that does not intersect $W\cap\{x\in\mathbb{Z}^d\,:\,x_1\ge0\}$ and that this coset is $B_p$. Therefore, by knowing only the restriction of $W$ to the half-space, we can recover all $B_p$'s, hence $W$. 
Let us put the rigidity of this model into a wider perspective.

An important dichotomy in mathematics is the one between structured and disordered systems. It plays a role in diverse fields of mathematics \cite{tao-blog}. Structured systems can be studied fruitfully by using their rigidity, possibly via algebraic methods. As for disordered systems, they connect to probability theory, where other tools are available. It is an important mathematical fact that, in several contexts, any system can be decomposed into structured and disordered parts. Szemer\'edi's Regularity Lemma is such a statement for graphs \cite{szem-lemma}, and the Furstenberg--Zimmer Theorem does so very cleanly for ergodic actions of $\mathbb{Z}$ or $\mathbb{Z}^d$ --- see \cite{zimmer, furstenberg, furstenberg-book}. Once we have extracted all the order there is in a system, it is disordered enough to be amenable to a probabilistic approach. One could say that there is no no man's land between probability and algebra.

This paves the way for a strategy to handle arbitrary systems: deal with structured systems and with disordered ones, and then use the decomposition theorem to carefully reduce the general case to these cases. In additive combinatorics, Szemer\'edi's Theorem, which states that any subset of $\mathbb{N}$ that has a positive upper-density contains arbitrarily long arithmetic progressions, has been proved in this way --- see both \cite{szemeredi} and \cite{furstenberg}.

In the context of the Furstenberg--Zimmer Theorem, the technical word for ``structured'' is ``compact'' (or ``almost periodic''), and that for ``disordered'' is ``weakly mixing''. We shall not need details here. The interested reader may learn about this in \cite{tao-blog, tao-lecture, furstenberg-book} or read a few introductory words in Appendix~\ref{appen}.

Doing percolation on $\mathbb{Z}^d$ means to pick a random subset of $\mathbb{Z}^d$ and study the ensuing clusters. One may pick any $\mathbb{Z}^d$-invariant probability measure on $\mathscr{P}(\mathbb{Z}^d)$. To the best of our knowledge, our investigation is, together with \cite{martineau}, the first one to do so for a probability measure defining a \emph{compact} system.\footnote{It is indeed compact as it is conjugate to a system as in Exercise~\ref{exo} with $K=\prod_p(\mathbb{Z}^d/p\mathbb{Z}^d)$, except that the group acting is $\mathbb{Z}^d$ instead of $\mathbb Z$.} All percolation articles we know of cover either general properties of percolations, or noncompact examples. Still, independently of these Furstenberg--Zimmer motivations, a few non weakly-mixing examples have been studied: see notably~\cite{jonasson, hoffman, pete, brochette, marcelo-vladas, ksv, hsst}. For such percolations, the ``compact part'' of their Furstenberg--Zimmer decomposition is nontrivial. 
In Appendix~\ref{appen:other}, we pick an example and explain without technicalities why it has some compact (a.k.a~almost periodic) content.

As it may be the case that a comprehensive understanding of percolation would have to go through a specific understanding of both the weakly mixing case and the compact one, it seems worthwhile to lead a proper investigation of the compact case. The present paper and \cite{martineau} offer a starting point for such an investigation, by providing a simple archetypal case of study. 

Studying percolative properties of the random coprime colouring is a nice playground for other directions of study as well. First, it offers a smooth interplay with number theory. As for percolation theory, it is one natural model to consider that enjoys none of the following properties: insertion-tolerance, deletion-tolerance, the FKG inequality. See Section~\ref{sec:context-perco} for definitions and a proof. Due to the lack of insertion-tolerance (resp.~deletion-tolerance) for $W$, uniqueness of the infinite white (resp. black) cluster cannot be obtained via the usual Burton--Keane argument \cite{burton-keane}, and requires a specific treatment. This treatment is made possible by the fact that the random coprime colouring enjoys its own kind of nice properties that usual Bernoulli percolation does not have: see Remarks~\ref{rem:quotient} and \ref{rem:long}.

\subsection{Some usual properties that are \emph{not} enjoyed by the random coprime colouring}
\label{sec:context-perco}
Let $\mu$ be a probability measure on $\mathscr{P}(\Gamma)$. 
A subset $\mathcal{A}$ of $\mathscr{P}(\Gamma)$ is said to be \emph{increasing} if, for all $\omega\in \mathcal{A}$ and all $\omega' \supset \omega$, we have $\omega'\in\mathcal{A}$. Less formally, being increasing means that adding new vertices to a configuration can only help to fulfill $\mathcal{A}$.
Say that $\mu$ satisfies the \emph{FKG inequality} if for all increasing measurable $\mathcal{A},\mathcal{B}\subset \mathscr{P}(\Gamma)$, we have $\mu(\mathcal{A}\cap\mathcal{B})\ge \mu(\mathcal{A})\mu(\mathcal{B})$. The acronym FKG stands for Fortuin--Kasteleyn--Ginibre \cite{harris, fkg}.
We also say that $\mu$ satisfies a \emph{weak FKG inequality} if for all increasing measurable $\mathcal{A},\mathcal{B}\subset \mathscr{P}(\Gamma)$ of positive $\mu$-probability, we have $\mu(\mathcal{A}\cap\mathcal{B})>0$.

Say that the measure $\mu$ is \emph{insertion-tolerant} if, for every $v\in\Gamma$, the conditional probability $\mu(v\in \omega\,|\,\omega\setminus\{v\})$ is $\mu$-almost surely positive. Likewise, say that $\mu$ is \emph{deletion-tolerant} if, for every $v\in\Gamma$, the conditional probability $\mu(v\notin \omega\,|\,\omega\setminus\{v\})$ is $\mu$-almost surely positive.
Insertion-tolerance and deletion-tolerance are useful to perform surgeries in finite regions \cite[Chapter~7]{lyons-peres}. For general background on percolation theory, the reader may pick any of \cite{bollobas-riordan, grimmett, lyons-peres, werner} or the lecture notes \cite{duminil}.

\begin{enonce}{Fact}
    Whenever $d\ge 2$, the distribution of $W$ does not satisfy any of the following properties: the FKG inequality (even in the weak form), insertion-tolerance, deletion-tolerance.
\end{enonce}

\begin{proof}
    By iteration, if weak FKG holds for two increasing events, then it also holds for $2^d$ of them. For all $v\in \{0,1\}^d$, let $\mathcal{A}_v:=\{\omega\,:\,v\in \omega\}$, which is increasing and measurable. For each of the four possible $v$, we have $\rho(\mathcal{A}_v)=\frac{1}{\zeta(d)}>0$, as $d\ge 2$. But their intersection has probability zero for $\rho$, as $B_2$ necessarily intersects $\{0,1\}^d$. Therefore, the probability measure $\rho$ does not satisfy the weak FKG inequality. In particular, it does not satisfy FKG.

    As for insertion-tolerance and deletion-tolerance, the reasoning explained at the beginning of Section~\ref{sec:context} explains that knowing $W$ in a half-space almost surely suffices to reconstruct $W$ exactly. In particular, knowing $W$ outside $\{v\}$ almost surely suffices to determine whether $v$ belongs to $W$ or not. Therefore, the conditional probability $\rho(v\in \omega\,|\,\omega\setminus\{v\})$ belongs $\rho$-almost surely to $\{0,1\}$. This is incompatible with both insertion and deletion-tolerance, as $\rho(\{\omega\,:\,v\in\omega\})=\frac{1}{\zeta(d)}\in (0,1)$.
\end{proof}

\section{Reduction to good crossing of rectangles}
\label{sec:reduction}

The purpose of this section is to establish Theorem~\ref{thm:usual}, provided some estimates hold. Before doing so, let us warm up by proving the easy part of Theorem~\ref{thm:usual}, namely the fact that there is at least one infinite white cluster when $d\ge3$.

\subsection{Existence of an infinite white cluster for $d\ge 3$}\label{sec:warm-up}
Let us work in dimension $d\ge 3$ and endow $\mathbb{Z}^d$ with its usual graph structure. It suffices to prove that, with positive probability, the infinite set $\mathbb\{0\}^2\times\mathbb{Z}^{d-2}$ contains no black element.
This event can be rephrased as ``for every prime $p$, the set $B_p$ does not interesect $\mathbb\{0\}^2\times\mathbb{Z}^{d-2}$''. For every $p$, exactly $p^d-p^{d-2}$ cosets are allowed for $B_p$ so, by independence, the probability of the event is $\prod_p\frac{p^d-p^{d-2}}{p^d}=\prod_p\left(1-p^{-2}\right)$, which is equal to $\frac{1}{\zeta(2)}$ by Eulerian product. We conclude by recalling that, as $2>1$, we have $\zeta(2)<\infty$ hence $\frac{1}{\zeta(2)}>0$.

\begin{rema}
    The easy proof above illustrates well how different this process is from Bernoulli site percolation, which corresponds to each element being white with probability $\alpha\in(0,1)$, black otherwise, independently of one another. In such a setup, the probability that a given set $A$ is fully white is $\alpha^{|A|}$. In particular, given any deterministic infinite set, the probability it is fully white is zero. This contrasts with what we have observed for our model with $A=\mathbb\{0\}^2\times\mathbb{Z}^{d-2}$.
\end{rema}

Reasoning as in the proof above with $d=2$ and the infinite set $\{0\}\times\mathbb{Z}$, we see that the probability that $\{0\}\times\mathbb{Z}$ contains no black element is $\frac{1}{\zeta(1)}=\frac{1}{\infty}=0$. Because of this, the 2-dimensional case will not be trivial to handle: instead of looking for infinite horizontal or vertical white lines, we shall look for \emph{long} white lines.

\subsection{The 2-dimensional case}\label{sec:2-dim} In this section, we assume that $d=2$. Our purpose is to prove that, under this additional assumption, the conclusion of Theorem~\ref{thm:usual} holds. The general case will then be deduced from this particular case.

\newcommand{\horizon}[6]{%
\begin{tikzpicture}%
    \draw[red] (0,0.1)--(0.45,0.1);%
    \draw[color=darkgray!50!gray, very thin] (0,0.04)--(0,0.16);%
    \draw[color=darkgray!50!gray, very thin] (0.45,0.04)--(0.45,0.16);%
    \node[red] at (#4,0.1) {$\scriptstyle #3$};%
    \node[color=darkgray!50!gray] at (0,#2) {$\scriptstyle #1$};%
    \node[color=darkgray!50!gray] at (0.45,#6) {$\scriptstyle #5$};%
\end{tikzpicture}%
}

\newcommand{\vertica}[6]{%
\begin{tikzpicture}%
    \draw[red] (0.1,0)--(0.1,0.38);%
    \draw[color=darkgray!61!gray] (0.04,0)--(0.16,0);%
    \draw[color=darkgray!61!gray] (0.04,0.38)--(0.16,0.38);%
    \node[red] at (0.1,#4) {$\scriptstyle #3$};%
    \node[color=darkgray!61!gray] at (#2,0.38) {$\scriptstyle #1$};%
    \node[color=darkgray!61!gray] at (#6,0) {$\scriptstyle #5$};%
\end{tikzpicture}%
}

\newcommand{\lengthrec}{0.5}
\newcommand{\heightrec}{0.32}
\newcommand{\halfheight}{0.16}

\newcommand{\horizontal}[8]{%
\begin{tikzpicture}%
    \draw[red] (0,\halfheight)--(\lengthrec,\halfheight);%
    \draw[color=darkgray!50!gray, very thin] (0,0)--(\lengthrec,0);%
    \draw[color=darkgray!50!gray, very thin] (0,\heightrec)--(\lengthrec,\heightrec);%
    \draw[color=darkgray!50!gray, very thin] (0,0)--(0,\heightrec);%
    \draw[color=darkgray!50!gray, very thin] (\lengthrec,0)--(\lengthrec,\heightrec);%
    \node[color=darkgray!50!gray] at (0,#2) {$\scriptstyle #1$};%
    \node[color=darkgray!50!gray] at (\lengthrec,#4) {$\scriptstyle #3$};%
    \node[color=darkgray!50!gray] at (#6,0) {$\scriptstyle #5$};%
    \node[color=darkgray!50!gray] at (#8,\heightrec) {$\scriptstyle #7$};%
\end{tikzpicture}%
}

\newcommand{\lengthvert}{0.32}
\newcommand{\halflength}{0.16}
\newcommand{\heightvert}{0.5}

\newcommand{\vertical}[8]{%
\begin{tikzpicture}%
    \draw[red] (\halflength,0)--(\halflength,\heightvert);%
    \draw[color=darkgray!50!gray, very thin] (0,0)--(\lengthvert,0);%
    \draw[color=darkgray!50!gray, very thin] (0,\heightvert)--(\lengthvert,\heightvert);%
    \draw[color=darkgray!50!gray, very thin] (0,0)--(0,\heightvert);%
    \draw[color=darkgray!50!gray, very thin] (\lengthvert,0)--(\lengthvert,\heightvert);%
    \node[color=darkgray!50!gray] at (0,#2) {$\scriptstyle #1$};%
    \node[color=darkgray!50!gray] at (\lengthvert,#4) {$\scriptstyle #3$};%
    \node[color=darkgray!50!gray] at (#6,0) {$\scriptstyle #5$};%
    \node[color=darkgray!50!gray] at (#8,\heightvert) {$\scriptstyle #7$};%
\end{tikzpicture}%
}

\newcommand{\annulus}[1]{%
\begin{tikzpicture}%
    \draw[red] (-0.24,-0.16)--(0.24,-0.16);%
    \draw[red] (-0.16,-0.24)--(-0.16,0.24);%
    \draw[red] (-0.24,0.16)--(0.24,0.16);%
    \draw[red] (0.16,-0.24)--(0.16,0.24);%
    \draw[color=darkgray!50!gray, very thin] (-0.24,-0.24)--(-0.24,0.24)--(0.24,0.24)--(0.24,-0.24)--(-0.24,-0.24);%
    \draw[color=darkgray!50!gray, very thin] (-0.08,-0.08)--(-0.08,0.08)--(0.08,0.08)--(0.08,-0.08)--(-0.08,-0.08);%
    \node[color=darkgray!50!gray] at (0.48,0.13) {$\scriptstyle#1$};%
\end{tikzpicture}%
}

In Figure~\ref{fig:simu}, we can see horizontal and vertical white lines fully crossing the figure. As we shall see, this is no accident. Let us introduce some definitions to make use of this observation.
In the forthcoming definitions, the elements denoted by $i$, $i_1$, $i_2$, $j$, $j_1$ and $j_2$ are always taken in $\mathbb{Z}$. We use the colour red to indicate white lines: using white would be misleading given that a blank page is already white.

Given $j$, $i_1$ and $i_2$ such that $i_1\le i_2$, we write $\raisebox{-0.25cm}{\horizon{i_1}{-0.16}{j}{-0.18}{i_2}{-0.16}}$ for the event ``all elements of $\llbracket i_1,i_2\rrbracket\times\{j\}$ are white''.
For $i$, $j_1$ and $j_2$ satisfying $j_1\le j_2$, let $\raisebox{-0.37cm}{\vertica{j_2}{-0.12}{i}{-0.16}{j_1}{-0.12}}$ be the event ``all elements of $\{i\}\times\llbracket j_1,j_2\rrbracket$ are white''.

Given $i_1$, $i_2$, $j_1$ and $j_2$ satisfying $i_1\le i_2$ and $j_1\le j_2$, let us introduce two events, corresponding a rectangle being crossed by a horizontal (resp. vertical) white line: let $\raisebox{-0.40cm}{\horizontal{~i_1}{-0.18}{i_2}{-0.18}{j_1}{-0.18}{j_2}{-0.18}}$ be the event ``there is $j\in\llbracket j_1,j_2\rrbracket$ such that $\raisebox{-0.25cm}{\horizon{i_1}{-0.16}{j}{-0.18}{i_2}{-0.16}}$'', and let $\raisebox{-0.43cm}{\vertical{~i_1}{-0.18}{~i_2}{-0.18}{j_1}{-0.18}{j_2}{-0.18}}$ be the event ``there is $i\in\llbracket i_1,i_2\rrbracket$ such that $\raisebox{-0.37cm}{\vertica{j_2}{-0.12}{i}{-0.16}{j_1}{-0.12}}$''. These events play a prominent role in this paper. We may write them respectively $\raisebox{-0.40cm}{\horizontal{~i_1}{0.5}{i_2}{-0.18}{j_1}{-0.18}{j_2}{-0.18}}$ or $\raisebox{-0.43cm}{\vertical{~i_1}{0.68}{~i_2}{-0.18}{j_1}{-0.18}{j_2}{-0.18}}$ to prevent typographical collisions between the expressions for $i_1$ and $i_2$.

By invariance of our probability distribution under rotation by $\frac\pi 2$, the probability of $\raisebox{-0.40cm}{\horizontal{~i_1}{-0.18}{i_2}{-0.18}{j_1}{-0.18}{j_2}{-0.18}}$ equals that of $\raisebox{-0.45cm}{\vertical{~j_1}{-0.18}{~j_2}{-0.18}{i_1}{-0.18}{i_2}{-0.18}}$. Because of this, understanding the probability of $\raisebox{-0.40cm}{\horizontal{~i_1}{-0.18}{i_2}{-0.18}{j_1}{-0.18}{j_2}{-0.18}}$ will suffice. Observe that, by translation-invariance of our probability distribution, the probability of $\raisebox{-0.40cm}{\horizontal{~i_1}{-0.18}{i_2}{-0.18}{j_1}{-0.18}{j_2}{-0.18}}$  depends only on the width $w=i_2-i_1$ and the height $h=j_2-j_1$, and we denote it by $P(h,w)$. Having the hope that $P(h,w)$\label{intro:p} gets close to 1, it can be useful to introduce $R(h,w):=1-P(h,w)$.

Let us momentarily assume the following estimate and deduce that there is almost surely at least one infinite white cluster. This estimate will be established in Section~\ref{sec:goodcrossing}.

\begin{enonce}{Claim}
    \label{claim:sum}
    The quantity $\sum_{n=0}^\infty R(2^n,2^{n+1})$ is finite.
\end{enonce}

\begin{proof}[Proof of the existence of an infinite white cluster in the square lattice, assuming Claim~\ref{claim:sum}.]
    For every $n\ge0$, let $A_{n}$ denote the event $\raisebox{-0.40cm}{\horizontal{0}{-0.18}{~2^{n+1}}{-0.18}{0~}{-0.18}{2^n~}{-0.18}}$ if $n$ is even and  $\raisebox{-0.43cm}{\vertical{0}{-0.18}{2^n}{-0.18}{0~}{-0.18}{2^{n+1}~~\phantom{a}}{-0.18}}$ if $n$ is odd. Observe that for every $n$, the probability of $A_n$ is $P(2^n,2^{n+1})$. Thanks to Claim~\ref{claim:sum}, the Borel--Cantelli Lemma can be applied to the complements of the $A_n$, yielding that, almost surely, for all $n$ large enough, the event $A_n$ holds. Therefore, it is almost surely possible to build an infinite white staircase as in Figure~\ref{fig:staircase}, thus concluding the proof as only finitely many stairs are missing.
\end{proof}

\begin{figure}[h!!]
\centering
    \includegraphics[width=10cm]{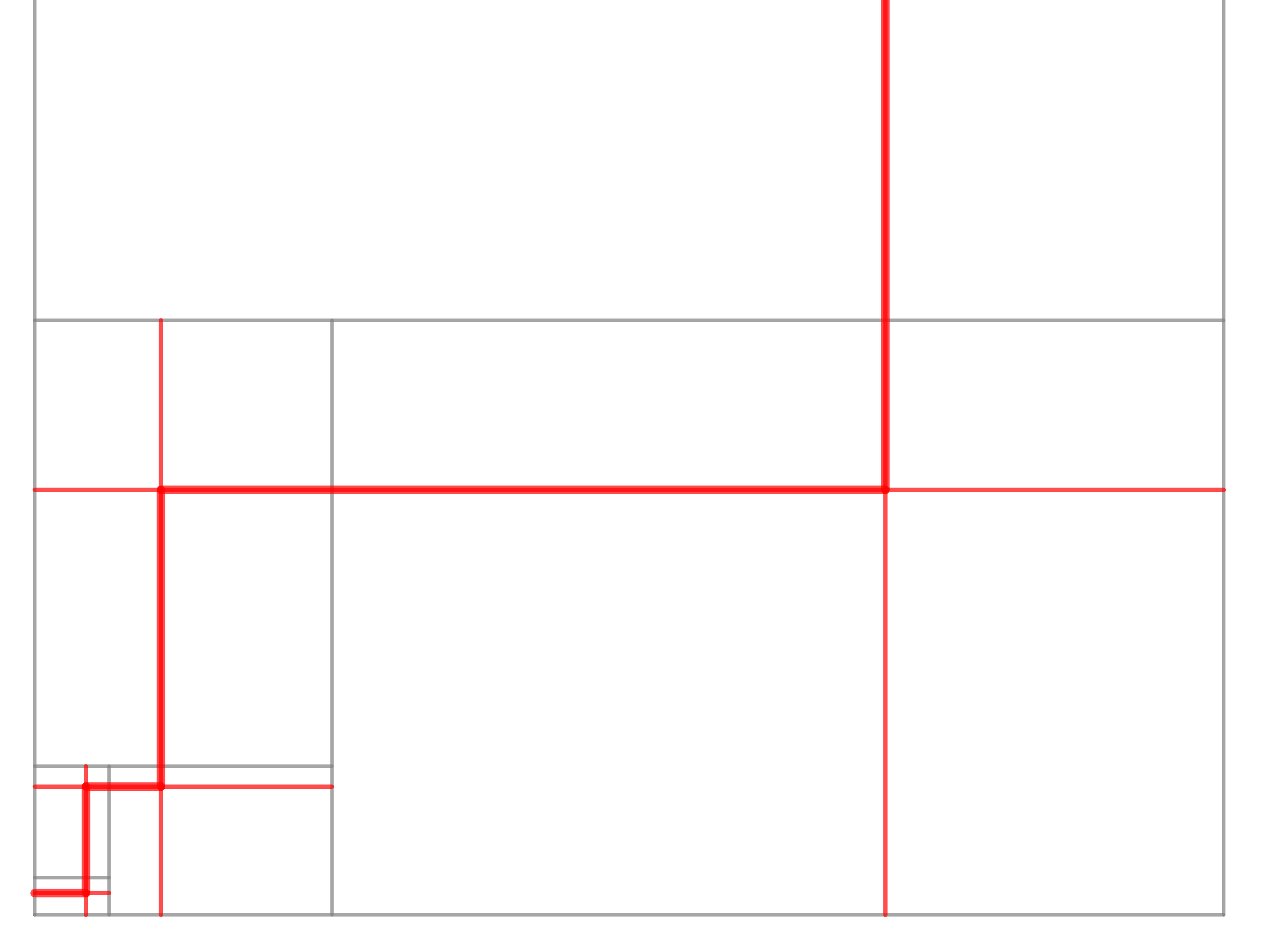}
    \caption{A depiction of the staircase argument, which produces an infinite path represented by a bold line.}
    \label{fig:staircase}
\end{figure}

\begin{rema}
    Note that, in the previous proof, we do not use the FKG inequality, only a simple Borel--Cantelli Lemma. As explained in Section~\ref{sec:context}, the process studied in the present paper does \emph{not} satisfy the FKG inequality.
\end{rema}

In order to solve the case of the square lattice, it remains to prove that, almost surely, there is no infinite black cluster and at most one infinite white cluster. To do so, we rely on the following claim, the proof of which is deferred to Section~\ref{sec:goodcrossing}.

\begin{enonce}{Claim}
    \label{claim:conv}
    The quantity $P(2\cdot3^n,2\cdot3^{n+1})$ converges to 1 when $n$ goes to infinity.
\end{enonce}

\begin{proof}[Proof of the inexistence of an infinite black cluster and the uniqueness of the infinite white cluster in the square lattice, assuming Claim~\ref{claim:conv}.]
    Let $n\ge1$ and let $k=3^n$. Denote by $\annulus{k}$ the event depicted on Figure~\ref{fig:annulus}. More precisely, we introduce
   \begin{equation}
   \label{eq:defannulusevent}
   \annulus{k~}:= \raisebox{-0.48cm}{\vertical{-k}{0.68}{-\frac k 3}{-0.18}{-k}{-0.28}{k}{-0.28}} \cap 
    \raisebox{-0.40cm}{\horizontal{-k}{-0.18}{k}{-0.18}{\frac k 3}{-0.27}{k}{-0.27}}
    \cap
   \raisebox{-0.40cm}{\horizontal{-k}{-0.18}{k}{-0.18}{-k}{-0.30}{-\frac k 3}{-0.30}}
   \cap \raisebox{-0.40cm}{\vertical{\frac k3}{0.76}{k}{-0.18}{-k}{-0.28}{k}{-0.28}}.
   \end{equation}
   Passing to the complement and using the union bound, we get
   \[\mathbb{P}(~~\raisebox{-0.1cm}{\annulus{3^n}}\hspace{-0.1cm})\ge 1- 4R(2\cdot3^{n-1},2\cdot3^{n}).\]
   By Claim~\ref{claim:conv}, the probability of $~~\annulus{3^n}\hspace{-0.2cm}$ converges to 1 as $n$ goes to infinity.

\begin{figure}[h!!]
\centering
    \includegraphics[width=10.1cm]{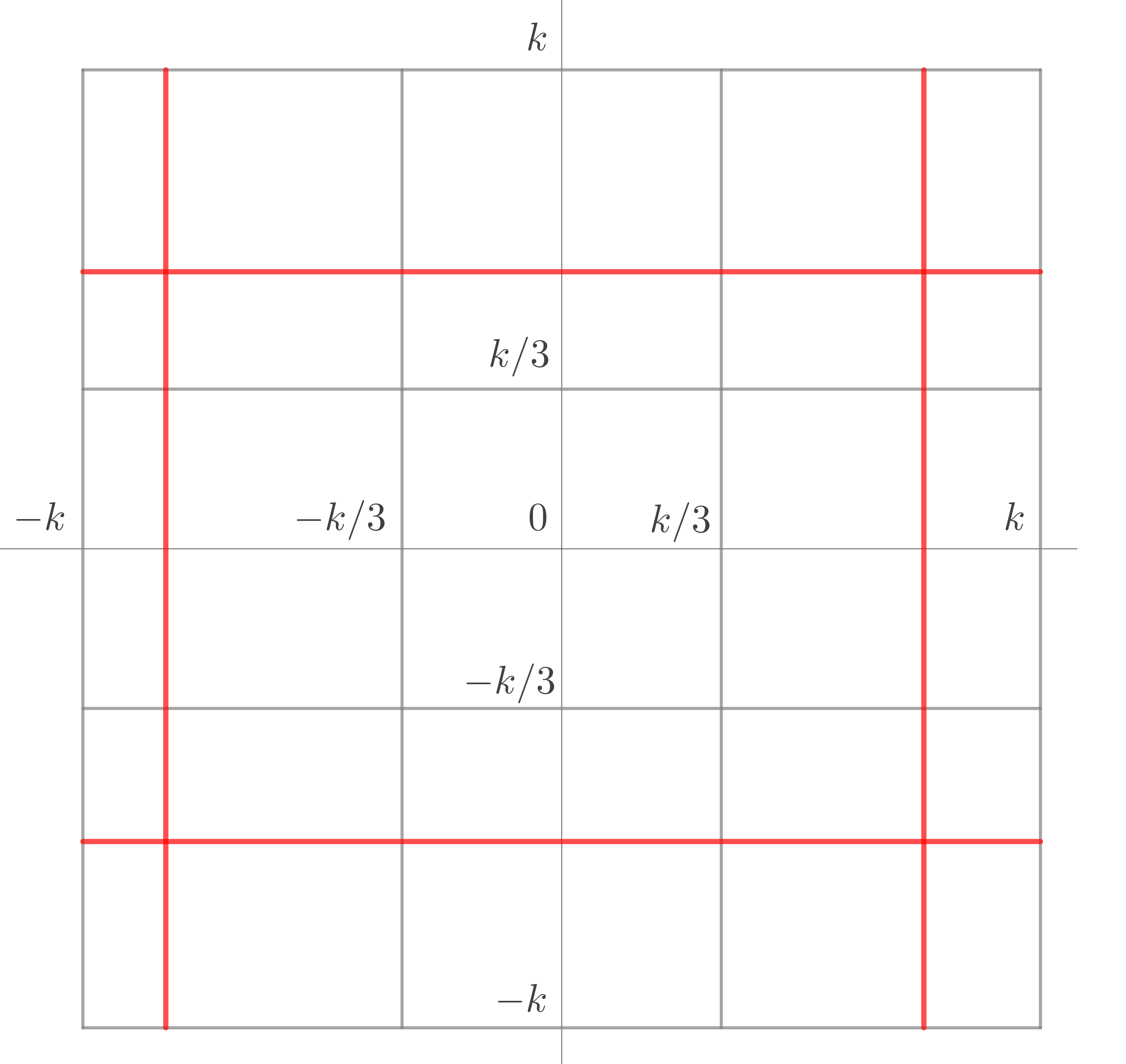}
    \caption{A visualisation of the event defined in \eqref{eq:defannulusevent}}
    \label{fig:annulus}
\end{figure}

    For $m\ge0$, let $E_m$ be the event defined as ``no infinite black cluster intersects $\llbracket -m,m\rrbracket^2$ and at most one infinite white cluster intersects $\llbracket -m,m\rrbracket^2$''.
   To conclude, it suffices to prove that, for every $m\ge0$, the event $E_m$ has probability 1. Let $m\ge 0$. Let $\eta<1$ and let us prove that $\mathbb{P}(E_m)\ge \eta$. As $\mathbb{P}(~~\raisebox{-0.1cm}{\annulus{3^n}}\hspace{-0.1cm})$ converges to 1, we can pick $n$ such that $\mathbb{P}(~~\raisebox{-0.1cm}{\annulus{3^n}}\hspace{-0.1cm})\ge\eta$ and $3^{n-1}> m$. Let us explain why $~~\annulus{3^n}\hspace{-0.2cm}\subset E_m$.

   \begin{figure}[h!!]
\centering
    \includegraphics[width=9cm]{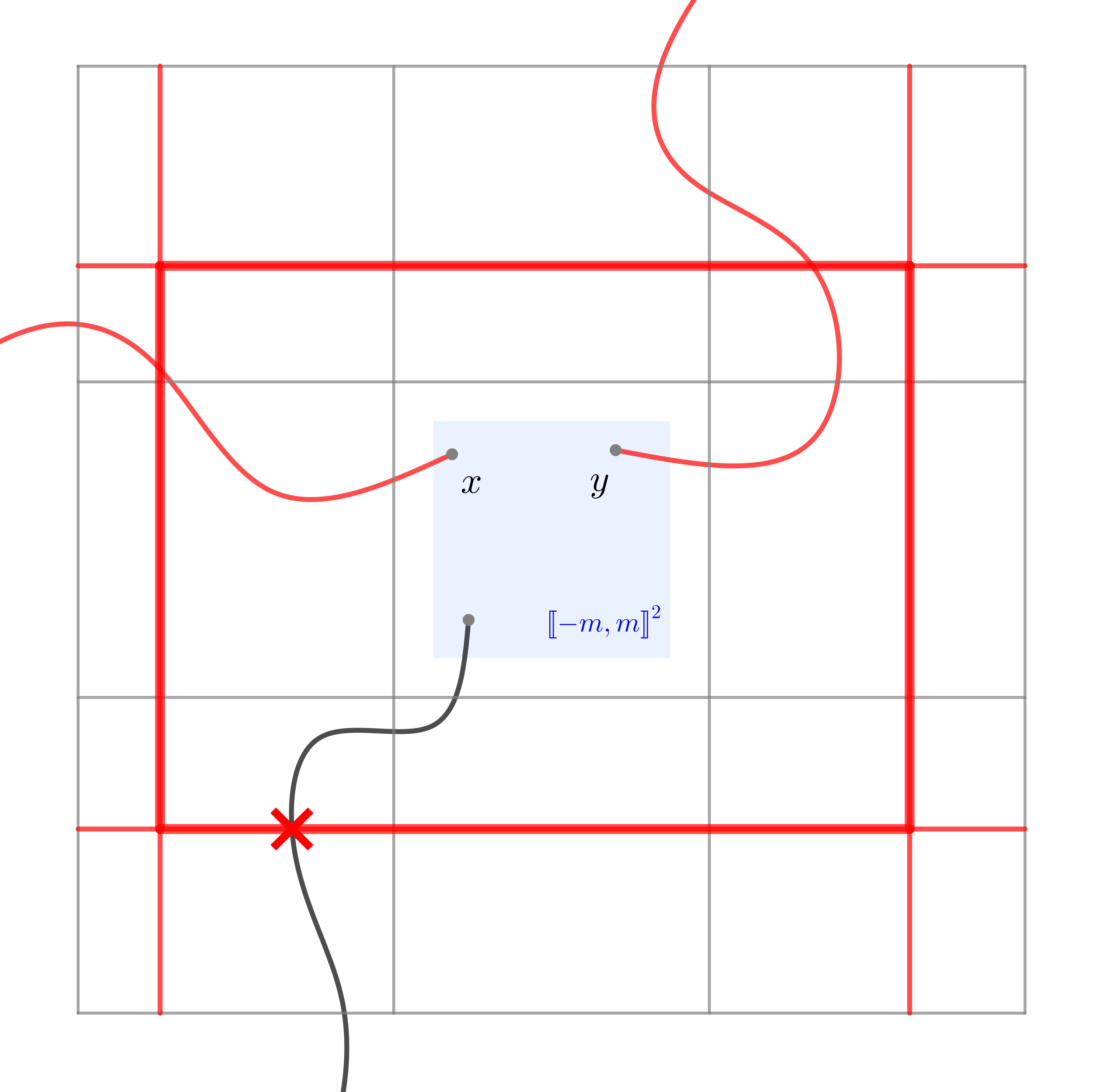}
    \caption{Given the white rectangle (depicted in red), it is impossible that an infinite black cluster intersects the blue zone. Likewise, any two infinite white clusters intersecting the blue area must in fact be the same cluster.}
    \label{fig:annubis}
\end{figure}
   
   When $~~\annulus{3^n}\hspace{-0.2cm}$ holds, there is a white rectangle surrounding $\{x\in\mathbb{Z}^2\,:\,\|x\|_\infty <3^{n-1}\}$: see in bold on Figure~\ref{fig:annubis}. As $3^{n-1}>m$, this rectangle surrounds $\llbracket-m,m\rrbracket^2$. Any infinite path starting in $\llbracket-m,m\rrbracket^2$ must intersect this rectangle and thus contain a white vertex. Therefore, when $~~\annulus{3^n}\hspace{-0.2cm}$ holds, no infinite black cluster can intersect $\llbracket-m,m\rrbracket^2$. Likewise, on this event, the following reasoning holds: given $x$ and $y$ in $\llbracket-m,m\rrbracket^2$, if they both belong to infinite white clusters, then both of them are connected to our rectangle by white paths, and as our rectangle is itself white and a connected set, $x$ and $y$ must belong to the same white cluster. Therefore, the inclusion $~~\annulus{3^n}\hspace{-0.2cm}\subset E_m$ holds, thus yielding $\mathbb{P}(E_m)\ge\mathbb{P}(~~\raisebox{-0.1cm}{\annulus{3^n}}\hspace{-0.1cm})\ge \eta$, as desired.
\end{proof}

\begin{rema}
    In the proof above, it suffices in the claim that the limsup is equal to 1 --- or even positive if one uses ergodicity. 
\end{rema}

\subsection{The general case}\label{sec:general} If we take Claims~\ref{claim:sum} and \ref{claim:conv} for granted, given Sections~\ref{sec:warm-up} and \ref{sec:2-dim}, the remaining task goes as follows: showing that, for the usual graph structure on $\mathbb{Z}^d$ and assuming $d\ge2$, there is almost surely no infinite black cluster and at most one infinite white cluster.

Let thus $d\ge2$ and $\mathbb{Z}^d$ be endowed with its usual graph structure. Let $i\in\llbracket2,d\rrbracket$ and let $\pi_{1i}:\mathbb{Z}^d\to\mathbb{Z}^2$ be defined by $x\longmapsto (x_1,x_i)$. For every prime $p$, let $B_p^{(1i)}:=\pi_{1i}(B_p)$. Then each $B_p^{(1i)}$ is a random coset of $p\mathbb{Z}^2$ chosen uniformly at random, and they are independent when $p$ varies.
Therefore, because of Section~\ref{sec:2-dim} and given some arbitrarily fixed $m\ge0$, we can almost surely pick a rectangle $\mathcal R_{1i}$ surrounding $\llbracket -m,m\rrbracket^2$ such that $\pi_{1i}(\bigcup_p B_p)$ does not intersect $\mathcal R_{1i}$. The set $\pi^{-1}(\mathcal R_{1i})$ contains only white vertices.

We now let $i$ vary. The set $\bigcup_{i=2}^d\pi^{-1}(\mathcal R_{1i})$ has the following three properties:
\begin{itemize}
    \item it contains only white vertices,
    \item any path from $\llbracket -m,m\rrbracket^d$ to infinity must intersect it,
    \item it is connected.
\end{itemize}
The second property holds because if a path $\kappa$ starting in $\llbracket -m,m\rrbracket^d$ never intersects $\bigcup_{i=2}^d\pi^{-1}(\mathcal R_{1i})$, then for every $i\ge2$, the sequence $(\kappa_1(n),\kappa_i(n))_{n\ge0}$ stays in the bounded component of the complement of $\mathcal R_{1,i}$, so that all components of $\kappa$ remain bounded.

The fact that $\bigcup_{i=2}^d\pi^{-1}(\mathcal R_{1i})$ has these three properties ends the proof exactly as in the end of Section~\ref{sec:2-dim}. No black path can go from $\llbracket-m,m\rrbracket^d$ to infinity as it would have to intersect the white set $\bigcup_{i=2}^d\pi^{-1}(\mathcal R_{1i})$. Besides, given any two infinite white paths starting in $\llbracket-m,m\rrbracket^d$, they must both intersect $\bigcup_{i=2}^d\pi^{-1}(\mathcal R_{1i})$ and we can then join them via a white path, as $\bigcup_{i=2}^d\pi^{-1}(\mathcal R_{1i})$ is connected and fully white. \qed

\vspace{0.3cm}

\begin{rema}
    Imagine that we could prove the claims for ``fat lines''. Namely, assume that for every $k\ge 0$, the claims hold if the event $\raisebox{-0.40cm}{\horizontal{~i_1}{-0.18}{i_2}{-0.18}{j_1}{-0.18}{j_2}{-0.18}}$ is taken to mean ``there  is some $j\in \llbracket j_1,j_2-k\rrbracket$ such that $\llbracket i_1,i_2\rrbracket\times \llbracket j,j+k\rrbracket$ contains only white vertices''. Then, the arguments of the present section would adapt and prove that, when $d\ge 2$, for all choices of $(\Gamma,S)$, there are almost surely a unique infinite white cluster and no infinite black cluster. This plan cannot work, though --- recall the counterexamples from page~\pageref{page:exist-black}. Actually, given $i_1<i_2$ and $j_1<j_2$, it is easy to see that the ``fat event'' has probability zero as soon as $k>0$. Indeed, the desired ``fat line'' would then have to contain a 2 by 2 square, which cannot be fully white because of $B_2$.
\end{rema}

\begin{rema}
    \label{rem:quotient}
    We have used that the image by $x\longmapsto (x_1,x_2)$ of the black vertices of the $d$-dimensional random coprime colouring has the same probability distribution as the black vertices of the 2-dimensional random coprime colouring. This connection between the same model considered on a graph and on a quotient of it is simpler than what happens for the usual independent Bernoulli percolation --- see \cite[proof of Theorem~1]{perco-beyond} and \cite{stochalift}.
\end{rema}

For Theorem~\ref{thm:usual} to be proven, it remains to establish the claims. The next section is devoted to this purpose.

\section{Rectangles are crossed with high probability}
\label{sec:goodcrossing}

In this section, we want to get useful estimates on the quantity $P(h,w):=\mathbb{P}(\raisebox{-0.40cm}{\horizontal{0}{-0.18}{w}{-0.21}{0}{-0.18}{h}{-0.18}})$ introduced on page~\pageref{intro:p}. As we shift from using probabilistically this quantity to studying it arithmetically, we change our notation.

Given positive integers $n$ and $x$, we consider $\mathbb{P}(\raisebox{-0.40cm}{\horizontal{1}{-0.18}{x}{-0.21}{1}{-0.18}{n}{-0.18}})=P(n-1,x-1)$. Shifting by 1 is useful so that studying $(n,x)$ indeed corresponds to $n$ rows and $x$ columns. As we are interested in how close $\mathbb{P}(\raisebox{-0.40cm}{\horizontal{1}{-0.18}{x}{-0.21}{1}{-0.18}{n}{-0.18}})$ is to 1, we introduce 
\[
r_n(x):=1-\mathbb{P}(\raisebox{-0.40cm}{\horizontal{1}{-0.18}{x}{-0.21}{1}{-0.18}{n}{-0.18}})=R(n-1,x-1).
\]

The goal of this section is to prove the following proposition. Claims~\ref{claim:sum} and \ref{claim:conv} are immediate corollaries of it. We use the notation $\mathcal O(\text{something})$ to mean ``some quantity the absolute value of which is upper bounded by a universal constant times $|\text{something}|$''.

\begin{prop}
    \label{prop:arith}
    For every $n>1$ and every $x\ge n$, we have $r_n(x) = \mathcal{O} \left( \tfrac{\log(x)}{n} \right)$.
\end{prop}

\begin{rema}
    \label{rem:long}
    Such a proposition would not hold at all if the black-and-white colouring was defined by Bernoulli site percolation, meaning if all vertices were coloured independently with some fixed probability $p\in(0,1)$ of being white. Indeed, the probability of $\raisebox{-0.40cm}{\horizontal{1}{-0.18}{x}{-0.21}{1}{-0.18}{n}{-0.18}}$ would then be equal to $1-(1-p^x)^n$ and upperbounded by $np^x$. The corresponding $r_n(x)$ would thus converge to 1 instead of 0 in the regime where $n$ goes to infinity and $x$ is proportional to $n$.
\end{rema}

We will prove Proposition~\ref{prop:arith} by a second moment method, for which we setup some notation. For every $k\in\mathbb{Z}$, let $X_{k,x}$ be the indicator function of the event $\raisebox{-0.25cm}{\horizon{1}{-0.16}{k}{-0.18}{x}{-0.16}}$. We also introduce $Z_{n,x}:=\sum_{k=1}^n X_{k,x}$ the random variable counting how many white horizontal lines cross the rectangle associated with $(n,x)$. Observe that the event $\raisebox{-0.40cm}{\horizontal{1}{-0.18}{x}{-0.21}{1}{-0.18}{n}{-0.18}}$ can be rephrased as ``$Z_{n,x}>0$''.
\newcommand{\E}{\mathbb{E}}

The Chebyshev inequality yields
\[
r_n(x) = \mathbb{P}(Z_{n,x}=0) \leq \frac{\mathrm{Var}(Z_{n,x})}{\mathbb{E}(Z_{n,x})^2} = \frac{\mathbb{E}(Z_{n,x}^2)}{\mathbb{E}(Z_{n,x})^2} - 1.
\]
We have $\E(Z_{n,x})=n\E(X_{1,x})$ and $\E(Z_{n,x}^2)=\sum_{1\le k,\ell\le n}\E(X_{k,x}X_{\ell,x})$, whence
\begin{equation}
\label{eq:cheby}
r_n(x)   \leq  \tfrac{1}{n^2} \sum_{1 \leq k,\ell \leq n} \frac{\E(X_{k,x}X_{\ell,x})}{\E(X_{1,x})^2} - 1.
\end{equation}
Our task is to estimate on average the quantity $\tfrac{\E(X_{k,x}X_{\ell,x})}{\E(X_{1,x})^2}$, considering $(n,x)$ to be fixed and the average to be over $(k,\ell)\in\llbracket 1,n\rrbracket^2$.
Let us introduce $f(x):=\E(X_{1,x})$. Having in mind the convention of page~\pageref{convention} regarding primes, we get the following formula for $f(x)$.
\begin{lemm}
    \label{lem:formula-f}
    For every $x$, we have $f(x) = \prod_p \left( 1 - \tfrac{\min(p,x)}{p^2} \right)$.
\end{lemm}

\begin{proof}
    Let $x$ be a positive integer. We have $f(x)=\mathbb{P}($\raisebox{-0.25cm}{\horizon{1}{-0.16}{1}{-0.18}{x}{-0.16}}$)$. The event $\raisebox{-0.25cm}{\horizon{1}{-0.16}{1}{-0.18}{x}{-0.16}}$ can be rephrased ``for every prime $p$, the coset $B_p$ does not intersect the line $\llbracket1,x\rrbracket\times\{1\}$''. By independence, $f(x)$ is the product over primes $p$ of $\mathbb{P}\left(B_p\text{ avoids } \llbracket1,x\rrbracket\times\{1\}\right)$. When we have $p\le x$, avoiding $\llbracket1,x\rrbracket\times\{1\}$ is the same as avoiding $\mathbb{Z}\times\{1\}$, so that exactly $p$ of the $p^2$ possible values for $B_p$ are forbidden: in this case, $\mathbb{P}\left(B_p\text{ avoids } \llbracket1,x\rrbracket\times\{1\}\right)$ is equal to $1-\tfrac{p}{p^2}$. When $p>x$, all values of $\llbracket1,x\rrbracket\times\{1\}$ are distinct when interpreted modulo $p$, so that exactly $x$ values for $B_p$ are forbidden: in this situation, the corresponding probability is $1-\frac{x}{p^2}$.
\end{proof}

The next step is to compute $\E(X_{k,x}X_{\ell,x})$. To this end, we proceed as in the proof of Lemma~\ref{lem:formula-f}. Given $k,\ell\in\mathbb{Z}$, a prime $p$ and a positive integer $x$, the probability that $B_p$ avoids $\llbracket1,x\rrbracket\times\{k,\ell\}$ is equal to the following quantity:
\[
g_{k,\ell,p}(x):=\begin{cases}
    1 - \frac{\min(p,x)}{p^2}& \text{ if } p\,\mid\,\ell-k, \\
    1 - \frac{2\min(p,x)}{p^2} & \text{ else},
\end{cases}
\]
where $a\,\mid\,b$ is a shorthand for ``$a$ divides $b$'', i.e.~``$b$ is a multiple of $a$''. As $X_{k,x}X_{\ell,x}$ is the indicator function of the event ``for every prime $p$, the random coset $B_p$ avoids $\llbracket1,x\rrbracket\times\{k,\ell\}$'', using the fact that the random variables $B_p$ are independent when $p$ ranges over primes, we obtain this formula:
\[
\E(X_{k,x}X_{\ell,x})=g_{k,\ell}(x):=\prod_p g_{k,\ell,p}(x).
\]

The prime $p=2$ plays a specific role: as soon as $x\ge 2$, for all $(k,\ell)$ such that $k-\ell$ is odd, we have $g_{k,\ell,2}(x)=0$ hence $g_{k,\ell}(x)=0$. The assumption $x\ge 2$ is free in Proposition~\ref{prop:arith}, as we have $x\ge n>1$. For this reason, in what comes next, $p=2$ plays a special role, and so is the case for the parity of $k-\ell$.

Let us immediately take this into account, and let us also isolate the diagonal case. By using $X_{1,x}^2 = X_{1,x}$, we can rewrite
\begin{eqnarray*}
\tfrac{1}{n^2} \sum_{1 \leq k,\ell \leq n} \frac{\E(X_{k,x} X_{\ell,x})}{\E(X_{1,x})^2} &  = &  \tfrac{2}{n^2} \sum_{\substack{1 \leq k<\ell\leq n \\ 2\,\mid\,\ell-k}} \frac{\E(X_{k,x} X_{\ell,x})}{\E(X_{1,x})^2} + \tfrac{1}{n} \frac{\E(X_{1,x}^2)}{\E(X_{1,x})^2} \\
 & = & \tfrac{2}{n^2} \sum_{\substack{1 \leq k<\ell\leq n \\ 2\,\mid\,\ell-k}} \frac{g_{k,\ell}(x)}{f(x)^2} + \frac{1}{n f(x)}\,.
\end{eqnarray*}
The quantity $g_{k,\ell}(x)$ depends on $(k,\ell)$ only through the value of $\ell-k$. Besides, for every $d\in\llbracket 1,n\rrbracket$, exactly $n-d$ couples $(k,\ell)$ satisfy both $1\le k<\ell\le n$ and $\ell-k=d$. We thus have
\begin{equation}
\label{eq:goodform}
\tfrac{1}{n^2} \sum_{1 \leq k,\ell \leq n} \frac{\E(X_{k,x} X_{\ell,x})}{\E(X_{1,x})^2}  = \tfrac{2}{n^2} \sum_{\substack{1 \leq d \leq n \\ 2\,\mid\,d}} (n-d) \frac{g_d(x)}{f(x)^2} + \frac{1}{n f(x)},
\end{equation}
where, for $d$ even, we set
\begin{equation*}
g_d(x) := g_{0,d}(x) = \tfrac{1}{2} \prod_{p>2} g_{0,d,p}(x).
\end{equation*}
Treating all factors depending on whether $p$ is equal to 2 or not, and whether it divides $d$ or not, we get the following expression for every even divisor $d$ of $n$:
\begin{eqnarray*}
\frac{g_d(x)}{f(x)^2}&  =&  2 \prod_{\substack{p>2 \\ p\,\mid\,d}} \left( 1 - \tfrac{1}{p} \right)^{-1} \prod_{\substack{p>2 \\ p\,\nmid\,d}} \frac{1 - \frac{2 \min(p,x)}{p^2}}{\left(1 - \frac{\min(p,x)}{p^2}\right)^2}\,,
\end{eqnarray*}
where we use in the first product that $d\le n \le x$.
Notice that the factors corresponding to $p>x$ all  appear in the second product, and that their product is 
\[
\prod_{p>x} \frac{1 - \frac{2x}{p^2}}{\left(1 - \frac{x}{p^2}\right)^2}\,.
\]
Expanding the logarithm, we obtain 
\[
 \log \left( \frac{1 - \frac{2x}{p^2}}{\left(1 - \frac{x}{p^2}\right)^2} \right) = - \tfrac{2x}{p^2} + \mathcal{O}\left(\tfrac{x^2}{p^4}\right) - 2 \left( - \tfrac{x}{p^2} + \mathcal{O}\left(\tfrac{x^2}{p^4}\right )\right) =  \mathcal{O}\left(\tfrac{x^2}{p^4}\right).
\]
By summing the logarithms over all primes $p>x$, we get
\[
\log\left(\prod_{p>x} \frac{1 - \frac{2x}{p^2}}{\left(1 - \frac{x}{p^2}\right)^2}\right)=\sum_{p>x}\mathcal{O}\left(\tfrac{x^2}{p^4}\right)=\mathcal{O}\left(x^2\cdot\sum_{k>x}\tfrac{1}{k^4}\right)=\mathcal{O}\left(\tfrac{x^2}{x^3}\right)=\mathcal{O}\left(\tfrac{1}{x}\right).
\]
We can thus write 
\begin{eqnarray*}
\frac{g_d(x)}{f(x)^2}&=&2 \left( 1 + \mathcal{O}\left(\tfrac{1}{x}\right) \right) \prod_{\substack{p>2 \\ p\,\mid\,d}} \left( 1 - \tfrac{1}{p} \right)^{-1} \prod_{\substack{2 <p \leq x \\ p\,\nmid\, d}} \frac{1 - \frac{2}{p}}{\left(1 - \frac{1}{p}\right)^2}\,\\
&=& 2 \left( 1 + \mathcal{O}\left(\tfrac{1}{x}\right) \right) \prod_{2 <p \leq x} \frac{1 - \tfrac{2}{p}}{\left(1 - \tfrac{1}{p}\right)^2} \, \prod_{\substack{p>2 \\ p\,\mid\,d}} \frac{1-\tfrac{1}{p}}{1 - \tfrac{2}{p}}\,.
\end{eqnarray*}
By handling the product over $p>x$ of $\tfrac{1 - 2/p}{\left(1 - 1/p\right)^2}$ as we handled the product of $\frac{1 - 2x/p^2}{\left(1 - x/p^2\right)^2}$, we get
\begin{equation}
\label{eq:gdxfx2}
\frac{g_d(x)}{f(x)^2}=2 \left( 1 + \mathcal{O}\left(\tfrac{1}{x}\right) \right) C_\infty\, \theta(d)\,,
\end{equation}
where we set
\[
 C_\infty := \prod_{p>2} \frac{1 - \tfrac2p}{\left(1-\tfrac1p\right)^2 } = \prod_{p>2} \left( 1 - \tfrac{1}{(p-1)^2} \right), \quad  \theta(d) := 
\prod_{\substack{p>2 \\ p\,\mid\,d}} \frac{1-\frac{1}{p}}{1 - \frac{2}{p}} . \]
Notice that $C_\infty$ is the twin prime constant, related to the twin prime conjecture \cite[section~22.20.2]{hardywright}.
Setting $d':=\tfrac{d}2$, we can write
\[
\theta(d) = \prod_{\substack{p>2 \\ p\,\mid\,d'}} \frac{p-1}{p-2} =  \prod_{\substack{p>2 \\ p\,\mid\,d'}} \left( 1 + \tfrac{1}{p-2} \right).
\]

It will sometimes be more convenient to write $\theta(d)$ as a sum of products rather than as a product of sums. To this end, let us introduce a useful function $\phi$. Its domain of definition is the set of all $k\ge1$ that are odd and squarefree. Recall that a positive integer $k$ is \defini{squarefree} (in short: sqf) if there is no $n>1$ such that $n^2\,\mid\, k$. The function $\phi$ we introduce maps any squarefree odd $k\ge1$ to $\prod_{p\,\mid\,k} \tfrac{1}{p-2}$. Then, for any even $d$, we have 
\[
\theta(d) = \sum_{\substack{k\,\mid\,d\\k \textrm{ odd, sqf}}} \phi(k).
\]
We can thus rewrite 
\begin{eqnarray}
\sum_{\substack{1 \leq d \leq n \\ 2\,\mid\,d}}(n-d) \,\theta(d) & = & \sum_{\substack{1 \leq d \leq n \notag\\ 2\,\mid\,d}}(n-d)   \sum_{\substack{k\,\mid\,d\\k \textrm{ odd, sqf}}} \phi(k) \\
	& = & \sum_{\substack{k=1 \\ k  \textrm{ odd, sqf}}}^{\lfloor n/2 \rfloor}  \phi(k) \sum_{d'=1}^{\lfloor n/(2k) \rfloor} (n-2kd')	\notag\\
	& \leq &  \sum_{\substack{k=1 \\ k  \textrm{ odd, sqf}}}^{\lfloor n/2 \rfloor}  \phi(k) \left( n\times \tfrac{n}{2k} - k \times \tfrac{n}{2k} \left( \tfrac{n}{2k} -1 \right) \right) \notag\\
	& \leq & \sum_{\substack{k=1 \\ k  \textrm{ odd, sqf}}}^{\lfloor n/2 \rfloor}  \phi(k) \left( \tfrac{n^2}{4k} + \tfrac{n}{2} \right) \notag\\
    \sum_{\substack{1 \leq d \leq n \\ 2\,\mid\,d}}(n-d) \,\theta(d) & \leq & \tfrac{n^2}{4}  \sum_{\substack{k=1 \\ k  \textrm{ odd, sqf}}}^{\lfloor n/2 \rfloor}  \frac{\phi(k)}{k} + \tfrac{n}{2} \sum_{\substack{k=1 \\ k  \textrm{ odd, sqf}}}^{\lfloor n/2 \rfloor} \phi(k).\label{eq:comparewithphi} 
\end{eqnarray}

Let us point out that
\begin{eqnarray}
	\sum_{\substack{k\ge1 \\ k  \textrm{ odd, sqf}}} \frac{\phi(k)}{k } & = & \prod_{p>2} \left( 1 + \tfrac{1}{p(p-2)} \right) \notag\\
	& = & \prod_{p>2} \frac{(p-1)^2}{p(p-2)}\notag\\
    \sum_{\substack{k\ge1 \\ k  \textrm{ odd, sqf}}} \frac{\phi(k)}{k }&=&  \frac{1}{C_\infty}\label{eq:sumphisurk}.
\end{eqnarray}

\begin{lemm}
	For every integer $N >1$, we have
	\begin{equation}
    \label{eq:sumphi}
	\sum_{\substack{k=1 \\ k  \textrm{ odd, sqf}}}^N \phi(k) = \mathcal{O}\left(\log N\right).
	\end{equation}
\end{lemm}

\begin{proof}
	We can get the following bounds:
	\begin{eqnarray*}
		0\le\log \left( \sum_{\substack{k=1 \\ k  \textrm{ odd, sqf}}}^N \phi(k) \right) & \leq & \log \left( \prod_{2<p \leq N} \left( 1 + \tfrac{1}{p-2} \right) \right) \\
		& \leq &  \sum_{2 <p \leq N} \log \left( 1 + \tfrac{1}{p-2} \right)  \\
		& \leq & \sum_{2 <p \leq N} \frac{1}{p-2} \\
		& \leq & \sum_{2 <p \leq N} \frac{1}{p} + \sum_{2 <p \leq N} \frac{2}{p(p-2)} \\
		& \leq &  \log \log N + C + o(1) + \tfrac32\,.
 	\end{eqnarray*}
	In the final upper bound, the $\tfrac32$ comes from the observation that $\sum_{i=3}^\infty\frac{2}{i(i-2)}=\tfrac{3}{2}$, while the other part is given by \cite[Theorem~427]{hardywright}, where the constant $C$ is the Meissel--Mertens constant.
	Taking the exponential of this inequality yields the lemma.
\end{proof}

Putting all computations together, we successively obtain
\begin{eqnarray*}
\tfrac{2}{n^2} \sum_{\substack{1 \leq d \leq n \\ 2\,\mid\,d}} \frac{(n-d) g_d(x)}{f(x)^2} & = & \tfrac{4}{n^2} \left( 1 + \mathcal{O}\left(\tfrac1x\right) \right) C_\infty \sum_{\substack{1 \leq d \leq n \\ 2\,\mid\,d}} (n-d)\, \theta(d)\qquad\quad\text{ by \eqref{eq:gdxfx2}}\\
& \leq & \tfrac{4}{n^2} \left( 1 + \mathcal{O}\left(\tfrac1x\right) \right) C_\infty   \sum_{\substack{k=1 \\ k  \textrm{ odd, sqf}}}^{\lfloor n/2 \rfloor}  \phi(k) \left( \tfrac{n^2}{4k} + \tfrac{n}{2} \right)\quad\text{ by \eqref{eq:comparewithphi}}\\
& \leq & \tfrac{4}{n^2} \left( 1 + \mathcal{O}\left(\tfrac1x\right) \right) C_\infty  \left( \frac{n^2}{4 C_\infty} + \mathcal{O}(n \log(n)) \right)\quad\text{by \eqref{eq:sumphisurk} and \eqref{eq:sumphi}},
\end{eqnarray*}
which gives
\begin{equation}
\label{eq:offdiag}
    \tfrac{2}{n^2} \sum_{\substack{1 \leq d \leq n \\ 2\,\mid\,d}} \frac{(n-d) g_d(x)}{f(x)^2}\le1 + \mathcal{O}\left(\tfrac1x\right) + \mathcal{O}\left(\tfrac{\log(n)}{n}\right).
\end{equation}
Finally, to bound the contribution of the diagonal terms, observe that
\begin{eqnarray*}
- \log \left( \prod_{p>x} \left( 1 - \tfrac{x}{p^2} \right) \right) & = & -\sum_{p>x} \log \left(1 - \tfrac{x}{p^2} \right) \\ & \leq & 2 \sum_{p>x} \frac{x}{p^2} \\
& \leq & 2
\end{eqnarray*}
 and recall that
\[
	\prod_{p \leq x} \left(1 - \tfrac{1}{p} \right)^{-1} = \mathcal{O}\left(\log(x) \right),
	\]
	by Mertens' third theorem \cite[Theorem~429]{hardywright}. These estimates yield 
	\begin{equation}
    \label{eq:diag}
	\frac{1}{nf(x)} = \mathcal{O} \left( \tfrac{\log(x)}{n} \right).
	\end{equation}
By combining \eqref{eq:cheby}, \eqref{eq:goodform}, \eqref{eq:offdiag} and \eqref{eq:diag}, and by using $x\ge n$, we get
\[
r_n(x) = \mathcal{O} \left( \tfrac{1}{x} \right) + \mathcal{O} \left( \tfrac{\log(n)}{n} \right) + \mathcal{O} \left( \tfrac{\log(x)}{n} \right) = \mathcal{O}  \left( \tfrac{\log(x)}{n} \right),
\]
thus proving Proposition~\ref{prop:arith}.\hfill\qed

\section{Generalisation to some interesting pairs $(\Gamma,S)$}
\label{sec:gen}

In previous sections, we investigated the percolative properties of the random coprime colouring for $\Gamma=\mathbb{Z}^d$ and $S=\{x\in\mathbb{Z}^d\,:\,\|x\|_1=1\}$. In this section, we cover other $(\Gamma,S)$ of interest.
As a warm-up, let us first prove that there is an infinite white cluster as soon as $d\ge 3$.

\subsection{Existence of an infinite white cluster when $d\ge3$}
\label{sec:warm-up-general}
Let us make the assumptions of Proposition~\ref{prop:exist-white} and prove its conclusion. Let $s$ be some element of $S$. Among all vectors of the form $\lambda s$ for $\lambda\in (0,\infty)$, let $u_d$ be the shortest one that belongs to $\Gamma$. We can pick $u_1,\dots,u_{d-1}\in \Gamma$ such that $(u_1,\dots,u_d)$ is a basis of the free $\mathbb{Z}$-module $\Gamma$. In these coordinates, the argument of Section~\ref{sec:warm-up} applies. Indeed, the image of $\mathbb\{0\}^2\times\mathbb{Z}^{d-2}$ by $x\longmapsto x_1u_1+\dots+x_du_d$ contains an infinite connected set, namely the infinite path $(0,s,2s,3s,\dots)$.\hfill\qed

\subsection{Statement of the general results}
\label{sec:gen-statements}
For the remaining of Section~\ref{sec:gen}, our task will be to prove Theorems~\ref{thm:examples}, \ref{thm:examples-black} and \ref{thm:spread-out}.
To do so, we find setups that allow to adapt the arguments of previous sections, and then we check that each of our $(\Gamma,S)$ of interest fits one of these setups. These setups are given by Theorems~\ref{thm:setup} and \ref{thm:setupblack}.

\begin{theo}
    \label{thm:setup}
    Let $d\ge 2$. Let $\Gamma$ be a lattice in $\mathbb{R}^d$ that is furthermore a {\color{col}subset of $\mathbb{Z}^d$}. In other words, let $\Gamma$ be a finite-index subgroup of $\mathbb{Z}^d$. Let $S$ be an admissible generating subset of $\Gamma$. The set $\Gamma$ is endowed with the structure of Cayley graph given by $S$.

    Assume that for every $i\in\llbracket 1,d\rrbracket$, the following holds:
    \begin{enumerate}
        \item\label{item:setup-connect} the set of all $x\in \Gamma$ such that $x_i=0$ is a connected subset of $\mathrm{Cayley}(\Gamma,S)$,
        \item\label{item:setup-douanes} for any pair $(x,z)$ of adjacent vertices in $\mathrm{Cayley}(\Gamma,S)$ satisfying $x_i<0<z_i$, there is some $y\in \Gamma$ that is adjacent to $x$ or $z$ and that satisfies $y_i=0$.
    \end{enumerate}

    Then, the number of infinite white clusters is almost surely equal to 1.
\end{theo}

Here is a way to make sense of the assumptions of Theorem~\ref{thm:setup}. Think of coordinate-hyperplanes as walls that can be used to construct boxes, such as rectangles or their higher-dimensional versions. Then, Assumption~\ref{item:setup-connect} states that the walls are connected. As for Assumption~\ref{item:setup-douanes}, it implies that any path going from one side of the wall to the other must get adjacent to the wall --- or, in other words, that the 1-neighbourhood of the wall does separate the two sides of the wall. We deem that these assumptions are rather natural, given the arguments used in Section~\ref{sec:reduction}.

Theorem~\ref{thm:setup} can be stated in a way that is more conceptual, less coordinate-dependent. The statement goes as follows.

\begin{theo}
    \label{thm:coordinate-free}
    Let $d\ge2$. Let $\Gamma$ be a lattice in $\mathbb{R}^d$. Let $S$ be an admissible generating subset of $\Gamma$. The set $\Gamma$ is endowed with the structure of Cayley graph given by $S$.

    Assume that there is a basis $(\varphi_1,\dots,\varphi_d)$ of linear forms on $\mathbb{R}^d$ such that, for every $i$, the following conditions hold:
    \begin{enumerate}
        \item we have $\varphi_i(\Gamma)\subset \mathbb{Z}$,
        \item the set of all $x\in \Gamma$ such that $\varphi_i(x)=0$ is a connected subset of $\mathrm{Cayley}(\Gamma,S)$,
        \item for any pair $(x,z)$ of adjacent vertices in $\mathrm{Cayley}(\Gamma,S)$ satisfying $\varphi_i(x)<0<\varphi_i(z)$, there is some $y\in \Gamma$ that is adjacent to $x$ or $z$ and that satisfies $\varphi_i(y)=0$.
    \end{enumerate}

    Then, the number of infinite white clusters is almost surely equal to 1.
\end{theo}

\begin{proof}[Proof of equivalence between Theorems~\ref{thm:setup} and \ref{thm:coordinate-free}]
Assume that Theorem~\ref{thm:coordinate-free} holds. Then Theorem~\ref{thm:setup} follows by taking $\Gamma$ as in Theorem~\ref{thm:setup} and applying Theorem~\ref{thm:coordinate-free} with $\varphi_i:x\longmapsto x_i$.

Conversely, assume that Theorem~\ref{thm:setup} holds. Take $\Gamma$ and $(\varphi_1,\dots,\varphi_d)$ as in Theorem~\ref{thm:coordinate-free}. Let $\Phi:x\longmapsto(\varphi_1(x),\dots,\varphi_d(x))$. To conclude, it now suffices to apply Theorem~\ref{thm:setup} to $\tilde\Gamma:=\Phi(\Gamma)$.
\end{proof}

Theorem~\ref{thm:setup} shall be used to prove Theorems~\ref{thm:examples} and \ref{thm:spread-out}, which cover very nice examples but only provides the number of infinite white clusters. By modifying the second assumption, it is possible to get the number of infinite black clusters as well. This comes at the cost of covering less lattices. Here is the statement, which entails Theorem~\ref{thm:examples-black}.

\begin{theo}
    \label{thm:setupblack}
    Let $d\ge 2$. Let $\Gamma$ be a finite-index subgroup of $\mathbb{Z}^d$. Let $S$ be an admissible generating subset of $\Gamma$. The set $\Gamma$ is endowed with the structure of Cayley graph given by $S$.

    Assume that for every $i\in\llbracket 1,d\rrbracket$, the following holds:
    \begin{enumerate}
        \item the set of all $x\in \Gamma$ such that $x_i=0$ is a connected subset of $\mathrm{Cayley}(\Gamma,S)$,
        \item for any pair $(x,z)$ of adjacent vertices in $\mathrm{Cayley}(\Gamma,S)$ satisfying $x_i\le0\le z_i$, $x_i=0$ or $z_i=0$.
    \end{enumerate}

    Then, the number of infinite white clusters is almost surely equal to 1, and that of infinite black clusters is almost surely equal to 0.
\end{theo}

Section~\ref{sec:interestinglattices} gathers diverse information on the triangular lattice, the $D_d$ lattice, the $E_8$ lattice and the Leech lattice. This includes definitions and context. Section~\ref{sec:proba-proof} is then dedicated to proving Theorem~\ref{thm:setup}, and Section~\ref{sec:proba-proof-noire} explains how to adapt the arguments of Section~\ref{sec:proba-proof} to get Theorem~\ref{thm:setupblack}.
But from there, how do we get Theorems~\ref{thm:examples}, \ref{thm:examples-black} and \ref{thm:spread-out}?

Regarding the case of the triangular lattice, simply observe that there is a linear automorphism of $\mathbb{R}^d$ mapping $(\Gamma,S)$ to $(\mathbb{Z}^2,\{ (\pm1,0),(0,\pm1),\pm(1,1)\})$, which indeed satisfies the assumptions of Theorem~\ref{thm:setupblack}.
Actually, by using this particular representation of the triangular lattice, one can simply run the proof of Section~\ref{sec:2-dim} and check that it works rather than resorting to Theorem~\ref{thm:setupblack}. Theorem~\ref{thm:spread-out} is an immediate corollary of Theorem~\ref{thm:setup}, and one could also prove it by rerunning the proof of Section~\ref{sec:reduction}.
The remaining cases are thus
$D_d$, $E_8$ and the Leech lattice. They are the reason why we prove Theorems~\ref{thm:setup} and \ref{thm:setupblack}.

It is the purpose of Section~\ref{sec:setup-proof} to prove that, indeed, $D_d$ satisfies the assumptions of Theorem~\ref{thm:setupblack} for $d\ge 3$, and that $E_8$ and the Leech lattice satisfy the assumptions of Theorem~\ref{thm:setup}.
As for the case of $D_2$, it turns out to be the same as $(\mathbb{Z}^2,\{(\pm1,0),(0,\pm1)\})$, up to linear automorphism, so that it holds by Section~\ref{sec:2-dim}.

\vspace{0.25cm}

Sections~\ref{sec:proba-proof} and \ref{sec:proba-proof-noire} can be read independently of Section~\ref{sec:interestinglattices}. Section~\ref{sec:setup-proof} only requires knowledge of the definition of the lattices we consider, which are given in Section~\ref{sec:interestinglattices-def}. 

\subsection{Remarkable lattices}
\label{sec:interestinglattices}
In this section, we provide the definitions of several remarkable lattices, recall some of their properties, and explain why they are interesting. A reference for this section is \cite{conwaysloane}.

For our purposes, as soon as a pair $(\Gamma,S)$ is taken care of, so is its image $(\Phi(\Gamma),\Phi(S))$ for any linear automorphism $\Phi$ of $\mathbb{R}^d$. As linear automorphisms act transitively on lattices, the reason why we care about some specific $\Gamma$ is to define an interesting $S$ by taking the set of nonzero vectors of minimal Euclidean length in $\Gamma$ --- in our examples, this set indeed turns out to be a generating set of $\Gamma$. Therefore, if $\tilde\Gamma=\Phi(\Gamma)$ for some similarity $\Phi$, that is to say the composition of a linear dilation with a linear isometry, then it is not harmful if what some authors call by a given name is $\Gamma$ or $\tilde\Gamma$. In what comes next, the names of the lattices are standard but you may encounter them in the literature twisted by some similarity, i.e.~expressed in different yet suitable coordinates.

\subsubsection{Definitions}
\label{sec:interestinglattices-def}
Identifying $\mathbb{C}$ with $\mathbb{R}^2$, the \defini{triangular lattice} is defined to be $\mathbb{Z}[e^{i\pi/3}]$. The nonzero elements of minimal norm are exactly $e^{ik\pi/3}$ for $k\in\mathbb{Z}/6\mathbb{Z}$. There are 6 such vectors and they generate the triangular lattice.

For $d\ge2$, the \defini{lattice $D_d$} is defined to be the set of all $x\in\mathbb{Z}^d$ such that $x_1+\dots+x_d$ is even. The minimal Euclidean norm of a nonzero element of $D_d$ is 2. The nonzero elements of minimal norm are exactly given by the $x\in\mathbb{Z}^d$ such that exactly two of its entries belong to $\{-1,1\}$ while the other ones are all zero.
There are ${d\choose 2}\cdot2^2=2d(d-1)$ such vectors and they generate the $D_d$ lattice. We introduce this general family of lattices because it is a standard family that we can fully handle but we mostly care because the $D_{\color{col}4}$ lattice is worthy of interest.

The \defini{lattice $E_8$} is defined to be the set of all $x\in\mathbb{Z}^8$ such that all $x_i$ have the same parity and $\sum_i x_i$ is a multiple of 4. The minimal Euclidean norm of a nonzero element of $E_8$ is $\sqrt{8}$. The nonzero elements of minimal norm are exactly given by the $x\in\mathbb{Z}^8$ of one of the following types:
\begin{itemize}
    \item either all $x_i$ belong to $\{-1,+1\}$,
    \item or all $x_i$ are equal to 0 except for exactly two of them, that furthermore belong to $\{-2,+2\}$.
\end{itemize}
There are $2^8+{8\choose 2}\cdot2^2=240$ such vectors and they generate the $E_8$ lattice.

The definition of the Leech lattice is more involved. We first need the so-called extended binary Golay code. We denote by $\mathbb{F}_2$ the field with 2 elements, which is the same as $\mathbb{Z}/2\mathbb{Z}$. Endow $\mathbb{F}_2^{24}$ with a distance by setting $d(w,w'):=\#\{i\,:\,w_i\neq w'_i\}$. It turns out that there is a 12-dimensional linear subspace of the $\mathbb{F}_2$-vector space $\mathbb{F}_2^{24}$ any two distinct elements of which are at distance at least 8 of each other.\footnote{Here is one explicit construction of the extended binary Golay code. Instead of working in $\mathbb{F}_2^{24}$, let us work in $\mathbb{F}_2^{V}\times\mathbb{F}_2^V$, where $V$ is the vertex-set of the icosahedron. Say that a pair $(x,y)\in\mathbb{F}_2^{V}\times\mathbb{F}_2^V$ belongs to the extended binary Golay code if for every vertex $v$ of the icosahedron, $y_v$ is equal to the sum over all $u$ not adjacent to $v$ (including $v$ itself) of $x_u$. It is a nice exercise to check that $(x,y)$ belongs to the extended binary Golay code if and only if $(y,x)$ belongs to the extended binary Golay code.} Furthermore, this linear subspace is unique up to permutations $w\longmapsto (w_{\sigma(1)},\dots,w_{\sigma(24)})$. From now on, fix one such linear subspace and call it the \defini{extended binary Golay code}. Given two elements of the extended binary Golay code, their distance is always an element of $\{0,8,12,16,24\}$. A subset of $\llbracket 1,24\rrbracket$ is called an \defini{octad} (resp. \defini{dodecad}) if its indicator function modulo 2 is an element of the extended binary Golay code that furthermore is at distance exactly 8 (resp. 12) from $(0,\dots,0)$. There is no need to give a name corresponding to the value 16, as corresponding $x$ can be referred to as complements of octads. There are 759 octads and $2\,576$ dodecads, and every dodecad can be expressed as the symmetric difference of two octads (in a highly non-unique way).

The bit of weight 0 of an element $x$ of $\mathbb Z$ is 0 of $x$ is even and 1 if $x$ is odd. Recursively, for $k>0$, if $x\in\mathbb{Z}$ and $b\in\{0,1\}$ is its bit of weight 0, then the bit of weight $k$ of $x$, is defined to be the bit of weight $k-1$ of $\frac{x-b}2$. In other words, we are defining the binary digits of $x$, in a way that is well behaved for negative integers.

The \defini{Leech lattice} is the set of all $x\in \mathbb{Z}^{24}$ satisfying all these conditions:
\begin{enumerate}
    \item\label{item:bit-zero} all $x_i$ have the same bit of weight 0,
    \item\label{item:bit-one} if we denote by $y$ the element of $\mathbb{F}_2^{24}$ defined by setting $y_i$ to be the {\footnotesize (congruence modulo 2 of the)} bit of weight 1 of $x_i$, then $y$ belongs to the extended binary Golay code,
    \item\label{item:bit-two} the sum of the bits of weight 2 of all $x_i$ has the same parity as $x_1$ {\footnotesize(hence as any $x_j$, because of Item~\ref{item:bit-zero})}.
\end{enumerate}

The minimal Euclidean norm of a nonzero element of the Leech lattice is $\sqrt{32}=4\sqrt{2}$. The nonzero elements of minimal norm are exactly given by the $x\in\mathbb{Z}^{24}$ of one of the following types:
\begin{itemize}
    \item choose any pair in $\llbracket 1,24\rrbracket$ and ask for $x$ to be zero outside the pair and to be $\pm 4$ on the pair, where the sign is chosen independently for each element of the pair,
    \item choose any octad and ask for $x$ to be zero outside the octad and to be $\pm2$ on the octad, where the signs are allowed to depend on the element of the octad under the constraint that there is an even number of minus signs,
    \item the vector $x$ can be written $(\varepsilon_1y_1,\dots,\varepsilon_{24}y_{24})$ where $\varepsilon$ is a vector of signs such that $(\mathds{1}_{\varepsilon_1=+},\dots,\mathds{1}_{\varepsilon_{24}=+})$ is an element of the extended binary Golay code and all $y_i$ are equal to 1 except for exactly one $y_i$ which is equal to $-3$.
\end{itemize}
There are $1\,104 + 97\,152 + 98\,304 = 196\,560$ such vectors and they generate the Leech lattice.

When studying the Leech lattice, we will make use of the Mathieu groups $M_{24}$ and $M_{12}$ --- see \cite[Chapter~10]{conwaysloane}. The Mathieu group $M_{24}$ is the group of all permutations $\sigma\in\mathfrak{S}_{24}$ that preserve the extended binary Golay code. In other words, it is the set of all $\sigma$ such that, for every $w\in \mathbb{F}_2^{24}$, the element $(w_{\sigma(1)},\dots,w_{\sigma(24)})$ belongs to extended binary Golay code if and only if so is the case of $w$. Recall that a group action $G\curvearrowright X$ is \defini{$n$-transitive} if the diagonal action of $G$ on $n$-tuples of distinct elements of $X$ is transitive. The Mathieu group $M_{24}$ acts 5-transitively on $\llbracket1,24\rrbracket$. It also acts transitively on the set of dodecads.

Pick some dodecad $\Delta$. The group $M_{12}(\Delta)$ is defined to be the set of all $\sigma\in M_{24}$ such that, for all $i$, we have $\sigma(i)\in \Delta$ if and only if $i\in \Delta$. Different choices of $\Delta$ provide isomorphic $M_{12}(\Delta)$. Even better, they are conjugate as subgroups of $M_{24}$, as $M_{24}$ acts transitively on dodecads. The group $M_{12}(\Delta)$ acts 5-transitively on $\Delta$. It also acts 5-transitively on the complement dodecad $\Delta':=\llbracket1,24\rrbracket\setminus \Delta$, as $M_{12}(\Delta)=M_{12}(\Delta')$.

Let us point out that these properties of 5-transitivity are exceptional. If $X$ is a finite set such that $|X|\ge 5$ and if $G\curvearrowright X$ is a faithful 5-transitive group action, then this action is isomorphic to one given in the following list {\small(and all actions in the list are indeed faithful and 5-transitive)}:
\begin{itemize}
    \item the usual action of $\mathfrak{S}_n$ on $\llbracket1,n\rrbracket$ for $n\ge 5$,
    \item the usual action of $\mathfrak{A}_n$ on $\llbracket1,n\rrbracket$ for $n\ge 7$,
    \item the usual action of $M_{24}$ on $\llbracket1,24\rrbracket$,
    \item the usual action of $M_{12}(\Delta)$ on $\Delta$.
\end{itemize}
We say that two group actions $G\curvearrowright X$ and $H\curvearrowright Y$ are isomorphic if there are a group isomorphism $\varphi : G\to H$ and a bijection $f:X\to Y$ such that for all $g\in G$ and $x\in X$, we have $f(g\cdot x)=\varphi(g)\cdot f(x)$.

\subsubsection{Why are these lattices remarkable?}
\label{sec:remarkable}
The problem of kissing spheres goes as follows. Let $\mathbb{R}^d$ be the usual Euclidean space: what is the maximal number of disjoint open balls of radius 1 that one can arrange in such a way that each of them is tangent to the ball of radius 1 centred at the origin {\small(we consider that a ball is not tangent to itself)}?

The answer is currently known exactly in dimensions 1, 2, 3, 4, 8 and 24. In dimension 2, the optimal number is attained by putting the centres according to the minimal nonzero vectors of the triangular lattice, properly rescaled so that these vectors have norm 2. In dimension 3 (resp. 4, 8, 24), the same holds with the $D_3$ (resp. $D_4$, $E_8$, Leech) lattice. In dimensions 1, 2, 8 and 24, the optimal configurations are ``rigid'' (meaning: unique up to linear isomorphism) while, in dimensions 3 and 4, ``there is room to let the spheres wiggle'' in optimal configurations.

Let us move on to the sphere packing problem. In the Euclidean space $\mathbb{R}^d$, a sphere packing is simply a collection of mutually disjoint unit open balls. Given a sphere packing, one can define its (upper) density as follows. Denoting by $A$ the union of all balls of the packing, the upper density is $\limsup_{r\to\infty}\frac{\mathrm{Vol}(A\cap B_r)}{\mathrm{Vol}(B_r)}$, where $B_r$ stands for the ball of radius $r$ centred at the origin. If the limit $\lim_{r\to\infty}\frac{\mathrm{Vol}(A\cap B_r)}{\mathrm{Vol}(B_r)}$ exists, then we call it the density of the packing. Solving the sphere packing problem in dimension $d$ means to find a sphere packing of $\mathbb{R}^d$ that admits a certain density $\delta$ and such that no sphere packing of $\mathbb{R}^d$ has an upper density strictly larger than $\delta$.

The answer to the sphere packing problem is currently known in dimensions~1, 2, 3, 8 and 24. For $d=1$ (resp. 2, 3, 8, 24), a sphere packing of maximal density is given by taking the set of centres to be $\mathbb{Z}$ (resp. the triangular lattice, the face-centred cubic lattice, the $E_8$ lattice, the Leech lattice), suitably dilated so that the minimal distance between two distinct points is exactly 2. The 2-dimensional case was proven in the first half of the twentieth century \cite{thue, toth}. The 3-dimensional case was solved by Hales in 1998 \cite{hales} --- see also \cite{hales-et-al}. The case $d=8$ was solved in 2016 by Viazovska \cite{viazovska}. The case $d=24$ was treated one week later \cite{viazovska-et-al} by adapting the techniques from \cite{viazovska}. The same team later proved that the $E_8$ and Leech lattices are ``universally optimal'' \cite{viazovska-et-al-bis}. This means that, for a wide class of functional-minimising problems, the configuration of points $E_8$ is optimal in dimension~8 and the Leech lattice is so in dimension~24. This vastly generalises \cite{viazovska, viazovska-et-al}.

Another reason to care about the triangular lattice, $D_4$, $E_8$ and Leech is that they have a richer group of isometries than the hypercubic lattice of matching dimension. As we shall see, $D_4$ and the Leech lattice are related to situations involving an \emph{exceptional} amount of symmetry.

In dimension $d\ge 5$, there are exactly \emph{three} convex regular polytopes, up to similarity: the cube, the cocube, and the regular simplex --- which is self-dual. In dimension~3, there are \emph{five} of them, the so-called Platonic solids: to the cube, the regular octahedron and the regular tetrahedron, we add the regular icosahedron and the regular dodecahedron. And in dimension~4, there are \emph{six} of them. On top of the ones existing in all dimensions, there are the 120-cell {\footnotesize(the 4-dimensional version of the dodecahedron)}, the 600-cell {\footnotesize(the 4-d version of the icosahedron)} and the 24-cell. The 24-cell thus stands out as one of a kind. Among all dimensions $d\ge 3$, it is the unique convex regular polytope that is self-dual and is not a simplex. It turns out that the 24-cell can be constructed as the convex hull of the set of nonzero vectors of minimal norm in $D_4$.

Before accounting for the exceptional amount of symmetry in the Leech lattice, let us first discuss the extended binary Golay code. The fact that any two distinct points are necessarily at distance 8 or more means that if we give you an element of it via a noisy channel, then you can recover it perfectly if at most 3 bits have been altered, and you can detect mistransmission if exactly 4 bits have been modified. The extended binary Golay code was indeed used as an error-correcting code by the interstellar probes Voyager~1 and Voyager~2.

The extended binary Golay code is an exceptionally symmetric object. We have used it to defined the Mathieu groups $M_{24}$ and $M_{12}$. By taking the action $M_{24}\curvearrowright \llbracket1,24\rrbracket$ and considering the stabiliser of a point (resp. a couple of distinct points, a couple dodecad/point), one defines the Mathieu group $M_{23}$ (resp. $M_{22}$, $M_{11}$). These five Mathieu groups are 5 of the 26 sporadic finite simple groups. Let us recall the meaning of ``sporadic'' in this context. The classification of finite simple groups states that the list of all finite simple groups is given by: $\mathbb{Z}/p\mathbb{Z}$ for $p$ prime, $\mathfrak{A}_n$ for $n\neq 4$, the so-called ``groups of Lie type'' {\footnotesize(regrouping 18 infinite families of groups)}, and a list of 26 ``outliers'' --- the so-called sporadic groups. With the extended binary Golay code, we catch 5 of these exceptional groups. 
By studying symmetries of the Leech lattice, Conway was able to discover 3 other sporadic groups. The Leech lattice was also used to recover 4 previously discovered non-Mathieu sporadic groups.

\subsection{Proof of Theorem~\ref{thm:setup}}
\label{sec:proba-proof}
By definition of our percolation process, we have, for every prime $p$, a uniformly distributed coset $B_p$ of $p\Gamma$ in $\Gamma$, and these cosets form an independent family of random variables. For every $p$, there is a unique random coset $\tilde{B}_p$ of $p\mathbb{Z}^d$ in $\mathbb{Z}^d$ such that $B_p\subset \tilde{B}_p$ --- more explicitly, we have $\tilde{B}_p=B_p+p\mathbb{Z}^d$. As the random variables $B_p$ form an independent family, so is the case of the random cosets $\tilde{B}_p$. But, for a given prime $p$, what is the distribution of $\tilde{B}_p$?

Let $p$ be a prime. Composing the inclusion map $\Gamma\to\mathbb{Z}^d$ with the reduction modulo $p$ map $\mathbb{Z}^d\to\mathbb{Z}^d/p\mathbb{Z}^d$, we get a map $\varphi_p:\Gamma\to\mathbb{Z}^d/p\mathbb{Z}^d$. As its kernel contains $p\Gamma$, this gives rise to a well-defined morphism $\psi_p:\Gamma/p\Gamma\to\mathbb{Z}^d/p\mathbb{Z}^d$. This is related to the discussion above, as we have $\tilde{B}_p=\psi_p(B_p)$. Let $\mathbf{G}_p$ denote the image of $\psi_p$, which is also the image of $\varphi_p$. As the distribution of $B_p$ is invariant under the action of $\Gamma/p\Gamma$, the distribution of the $\mathbf{G}_p$-valued random variable $\tilde{B}_p$ has to be $\mathbf{G}_p$-invariant. Therefore, $\tilde{B}_p$ is uniformly distributed on the finite group $\mathbf{G}_p$.

Let us now prove that, provided $p$ is taken large enough, we have $\mathbf{G}_p=\mathbb{Z}^d/p\mathbb{Z}^d$. Let $N$ denote the index of $\Gamma$ in $\mathbb{Z}^d$. As $\Gamma$ has index $N$ in $\mathbb{Z}^d$ and as the reduction map $\mathbb{Z}^d\to\mathbb{Z}^d/p\mathbb{Z}^d$ is onto, the image of $\varphi_p$ has index at most $N$ in $\mathbb{Z}^d/p\mathbb{Z}^d$. The index of $\mathbf{G}_p$ is thus at most $N$ in $\mathbb{Z}^d/p\mathbb{Z}^d$. But the index of any subgroup of $\mathbb{Z}^d/p\mathbb{Z}^d$ is always a power of $p$. If we assume that $p>N$, then the only possible value for the index is 1, thus proving that $\mathbf{G}_p=\mathbb{Z}^d/p\mathbb{Z}^d$.
We can summarise what we obtained as follows.

\begin{lemm}
\label{lem:gammatozd}
    Let $\Gamma$ be a subgroup of $\mathbb{Z}^d$ of finite index $N$. Let $(B_p)$ be a prime-indexed sequence of independent uniform cosets of $p\Gamma$ in $\Gamma$.  
    Let $\mathbf{G}_p$ be defined as above.
    Let $(B'_p)$ be a prime-indexed sequence of independent uniform cosets of $p\mathbb{Z}^d$ in $\mathbb{Z}^d$.

    Then, the distribution of $(B_p)$ is the same as the conditional distribution of $(B'_p)$ given the event ``for all $p\le N$, we have $B'_p\in\mathbf{G}_p$''.
    
    In particular, if a sequence of events has a probability that converges to 0 (resp. 1) for $(B'_p)$, then so is the case for $(B_p)$. Likewise, if the convergence to 0 or 1 occurs summably fast for $(B'_p)$, then so is the case for $(B_p)$.
\end{lemm}

Ultimately, our goal is to mimic the proof of Section~\ref{sec:reduction}.
Lemma~\ref{lem:gammatozd} will be useful to transfer some probabilistic understanding from $\mathbb{Z}^d$ to $\Gamma$. Let us move on to more geometric considerations.

For every $i\in\llbracket 1,d\rrbracket$, let $k_i$ be such that the image of the map $\Gamma\ni x\longmapsto x_i\in\mathbb{Z}$ is $k_i\mathbb{Z}$. Let $\Gamma':=\{(\tfrac{x_1}{k_1},\dots,\tfrac{x_d}{k_d})\,:\, x\in \Gamma\}$ and let $S':=\{(\tfrac{x_1}{k_1},\dots,\tfrac{x_d}{k_d})\,:\, x\in S\}$. The pair $(\Gamma',S')$ still satisfies the assumptions of Theorem~\ref{thm:setup} but also the additional property that, for every $i$, the map $\Gamma'\ni x\longmapsto x_i\in\mathbb{Z}$ is onto. Besides, the conclusion of Theorem~\ref{thm:setup} holds for $(\Gamma,S)$ if and only if it holds for $(\Gamma',S')$. {\bf It is therefore without loss of generality that, from now on, we assume that for every $i$, the map $\Gamma\ni x\longmapsto x_i\in\mathbb{Z}$ is onto.} This shall be used in Claims~\ref{claim:douanes} and \ref{claim:connect}.

To perform our geometric constructions, we will use the following notation. For every $(i,j)\in\llbracket 1,d\rrbracket^2$ satisfying $i\neq j$, every $\ell\in\mathbb{Z}$ and every $(a,b)\in\mathbb{Z}^2$ satisfying $a\le b$, let $H_{i,j}^\ell(a,b)$ be the set of all $x\in \Gamma$ such that $x_i=\ell$ and $a\le x_j\le b$. In other words, $H_{i,j}^\ell(a,b)$ is the preimage by the map $\Gamma\ni x\longmapsto (x_i,x_j)\in\mathbb{Z}^2$ of the vertical line segment $\{\ell\}\times\llbracket a,b\rrbracket$.

Let $c\ge0$ be a suitable constant --- more precisely, let it be a constant satisfying $c\ge \max_{s\in S}\|s\|_\infty$ and such that the statements of the forthcoming Claims~\ref{claim:douanes} and \ref{claim:connect} hold. For every $n\ge 1$, let $A_n$ denote the following event: ``for every $(i,j)\in\llbracket 1,d\rrbracket^2$ satisfying $i\neq j$, there are $\ell_{i,j}^-\in\llbracket -3n,-n\rrbracket$ and $\ell_{i,j}^+$ in $\llbracket n,3n\rrbracket$ such that all elements of $H_{i,j}^{\ell_{i,j}^-}(-3n-2c,3n+2c)$ and $H_{i,j}^{\ell_{i,j}^+}(-3n-2c,3n+2c)$ are white''. See Figure~\ref{fig:fatannulus}.

\begin{figure}[h!!]
    \centering
    \includegraphics[width=12.06cm]{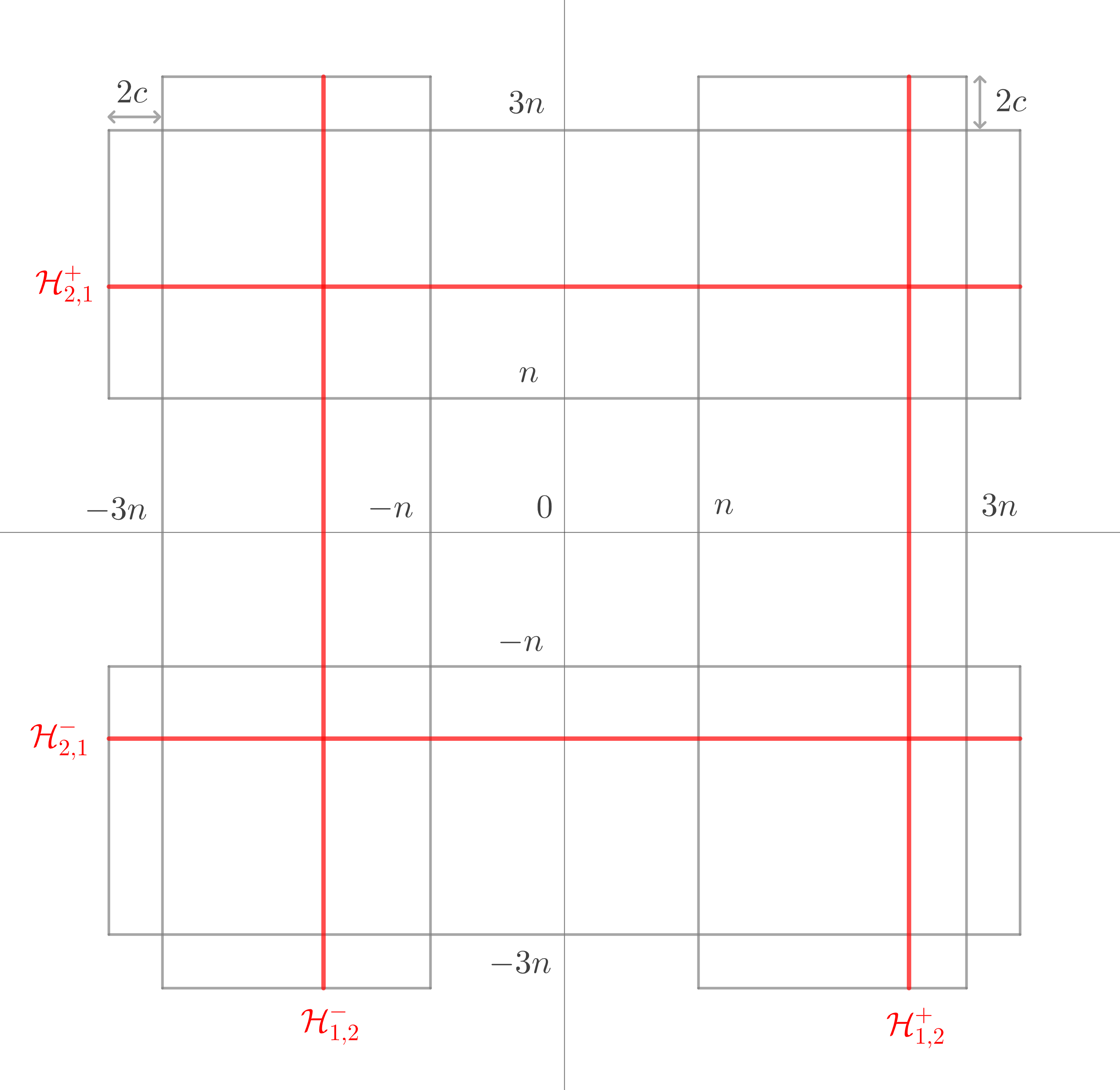}
    \caption{A depiction of the event $A_n$ in the case $d=2$. We write $\mathcal{H}_{i,j}^\varepsilon$ as a shorthand for $H_{i,j}^{\ell_{i,j}^\varepsilon}(-3n-2c,3n+2c)$.}
    \label{fig:fatannulus}
\end{figure}

\begin{lemm}
\label{lem:conv1}
    The probability of $A_n$ converges to $1$ as $n$ goes to infinity.
\end{lemm}

\begin{proof}
    The case $\Gamma=\mathbb{Z}^2$ follows from Proposition~\ref{prop:arith}.
   Let us now assume that $\Gamma=\mathbb{Z}^d$ for some $d\ge 2$. Let $(i,j)$ satisfying $i\neq j$ and $\pi:\mathbb{Z}^d \ni x\longmapsto (x_i,x_j)$. For every prime $p$, the random coset $\pi(B_p)$ is a uniform coset of $p\mathbb{Z}^2$ in $\mathbb{Z}^2$, and they are independent when $p$ varies. Therefore, as in the beginning of Section~\ref{sec:general}, the case $\Gamma=\mathbb{Z}^d$ follows from the case $\Gamma=\mathbb{Z}^2$. The general case follows from that of $\mathbb{Z}^d$ and Lemma~\ref{lem:gammatozd}.
\end{proof}

\subsubsection{Proof that there is at most one infinite white cluster}\label{sec:unique}
Let $m\ge 1$ be arbitrary. Using Lemma~\ref{lem:conv1} and proceeding as in Section~\ref{sec:general}, it suffices to prove the following lemma.
\begin{lemm}
    \label{lem:unique}
    This lemma is meant to be read within the proof of Theorem~\ref{thm:setup}. In particular, we do not restate the assumptions of this theorem, we assume that for all $i$, the map $\Gamma\ni x\longmapsto x_i\in\mathbb{Z}$ is onto, and $c$ stands for a well-chosen constant.
    
    Let $n>\max(m,c)$ and assume that for every $(i,j)$ satisfying $i\neq j$, we are given $\ell_{i,j}^-\in \llbracket -3n,-n\rrbracket$ and $\ell_{i,j}^+\in\llbracket n,3n\rrbracket$ such that all elements of $H_{i,j}^{\ell_{i,j}^-}(-3n-2c,3n+2c)$ and $H_{i,j}^{\ell_{i,j}^+}(-3n-2c,3n+2c)$ are white.
    
    Then the set $\mathcal{H}:=\bigcup_{i,j,\varepsilon}H_{i,j}^{\ell_{i,j}^\varepsilon}(-3n-c,3n+c)$ has the following three properties:
\begin{enumerate}
    \item  it contains only white vertices,\label{item:obvious}
    \item  any path from $\llbracket -m,m\rrbracket^d$ to infinity must either intersect it or get adjacent to it,\label{item:adja}
    \item for any two vertices of $\mathcal{H}$, there is a white path connecting them --- we allow paths exiting $\mathcal{H}$.\label{item:connect}
\end{enumerate}
\end{lemm}
Lemma~\ref{lem:unique} indeed guarantees uniqueness of the infinite white cluster, as two adjacent white vertices necessarily belong to the same white cluster. On the contrary, this condition does not suffice to prevent existence of infinite black clusters.

\begin{proof}[Proof of Lemma~\ref{lem:unique}]
Let us make the assumptions of Lemma~\ref{lem:unique} and establish the three desired properties. Item~\ref{item:obvious} holds by assumption as, for all $(i,j,\varepsilon)$, we have  the inclusion $H_{i,j}^{\ell_{i,j}^\varepsilon}(-3n-c,3n+c)\subset H_{i,j}^{\ell_{i,j}^\varepsilon}(-3n-2c,3n+2c)$.

Now, focus on Item~\ref{item:adja}. Let $\kappa$ be a path from $\llbracket-m,m\rrbracket^d$ to infinity. Let us consider the smallest possible $k$ such that there is some $(i,j,\varepsilon)$ satisfying $\varepsilon\kappa(k)_i\ge \varepsilon\ell_{i,j}^\varepsilon$. Such a finite $k$ has to exist as the path $\kappa$ goes to infinity, and $k$ cannot be zero as $n> m$. We fix $(i_0,j_0,\varepsilon_0)$ such that $\varepsilon_0\kappa(k)_{i_0}\ge \varepsilon_0\ell_{i_0,j_0}^{\varepsilon_0}$.
Let $x:=\kappa(k-1)$. By minimality of $k$, we know that $\varepsilon_0 x_{i_0}<\varepsilon_0\ell_{i_0,j_0}^{\varepsilon_0}$. Still by minimality of $k$, we have $-3n\le \ell_{j_0,i_0}^{-}< x_{j_0}<\ell_{j_0,i_0}^{+}\le 3n$.

\begin{enonce}{Claim}
    \label{claim:douanes}
    Recall that, on top of the assumptions of Theorem~\ref{thm:setup}, we assume that for all $i$, the map $\Gamma\ni x\longmapsto x_i\in\mathbb{Z}$ is onto. For all $c\ge 2\max_{s\in S}\|s\|_\infty$, the following statement holds:

    Let $(i,j)\in\llbracket 1,d\rrbracket^2$ be such that $i\neq j$. Let $\ell\in \mathbb{Z}$ and let $(a,b)\in\mathbb{Z}^2$ be such that $a\le b$.
    Let $(x,z)$ be a pair of adjacent vertices in $\mathrm{Cayley}(\Gamma,S)$ satisfying $x_i<\ell<z_i$ or $z_i<\ell<x_i$. Further assume that $a\le x_j\le b$.
    
    Then there is some $y\in H_{i,j}^\ell(a-c,b+c)$ that is adjacent to $x$ or $z$ and that satisfies $y_i=\ell$.
\end{enonce}

The easy proof of Claim~\ref{claim:douanes}, which relies on the second assumption of Theorem~\ref{thm:setup}, is deferred to Section~\ref{sec:douanes}. For now, let us use it to conclude the proof of Item~\ref{item:adja}. Take $c\ge 2\max_{s\in S}\|s\|_\infty$. Let $z:=\kappa(k)$ and $\ell:=\ell_{i_0,j_0}^{\varepsilon_0}$. If $\varepsilon_0z_{i_0}> \varepsilon_0\ell$, then we get the desired result by applying Claim~\ref{claim:douanes} with $a=-3n$ and $b=3n$. The remaining case is when $z_{i_0}=\ell$. In this case, we conclude by observing that $z\in H_{i_0,j_0}^\ell(-3n-c,3n+c)\subset\mathcal{H}$. Indeed, we have $z_{i_0}=\ell$ and, as $z$ is adjacent to $x$ which satisfies $-3n<x_{j_0}<3n$, the inequality $c\ge \max_{s\in S}\|s\|_\infty$ guarantees that $-3n-c<z_{j_0}<3n+c$.

At last, let us show Item~\ref{item:connect}. To this end, we rely on the following geometrical claim. Its proof, which relies on the first assumption of Theorem~\ref{thm:setup}, is postponed to Section~\ref{sec:connect}.

\begin{enonce}{Claim}
    \label{claim:connect}
    Recall that, on top of the assumptions of Theorem~\ref{thm:setup}, we assume that for all $i$, the map $\Gamma\ni x\longmapsto x_i\in\mathbb{Z}$ is onto. For all $c$ large enough, the following statement holds:

    Let $(i,j)\in\llbracket 1,d\rrbracket^2$ be such that $i\neq j$. Let $\ell\in \mathbb{Z}$ and let $(a,b)\in\mathbb{Z}^2$ be such that $a\le b$.
    
    Then, we have the following properties:
    \begin{itemize}
        \item any two points of $H_{i,j}^\ell(a-c,b+c)$ can be joined via a path lying in $H_{i,j}^\ell(a-2c,b+2c)$,
        \item it is possible to find two vertices $u$ and $v$ in $H_{i,j}^\ell(a-c,b+c)$ and a path $\kappa$ joining them in $H_{i,j}^\ell(a-c,b+c)$ in such a way that $u_j<a$, $v_j>b$ and the path $\kappa$ stays at $\|~\|_\infty$-distance at most $c$ from the segment $\ell e_i+[a,b]e_j$.
    \end{itemize}
\end{enonce}

Let $c\ge\max_{s\in S}\|s\|_\infty$ be such that the statements of Claims~\ref{claim:douanes} and \ref{claim:connect} hold. Applying the first part of the statement with $a=-3n$ and $b=3n$ gives that for every $(i,j,\varepsilon)$, any two vertices of $\mathcal{H}_{i,j}^\varepsilon(c):=H_{i,j}^{\ell_{i,j}^\varepsilon}(-3n-c,3n+c)$ can be joined by a path lying in the white set $\mathcal{H}_{i,j}^\varepsilon(2c)$. Therefore, it now suffices to prove that for any two $(i,j,\varepsilon)$ and $(i',j',\varepsilon')$, there is a white path joining $\mathcal{H}_{i,j}^{\varepsilon}(c)$ and $\mathcal{H}_{i',j'}^{\varepsilon'}(c)$. Proving this under the additional assumption that $i'=j$ suffices as well because, in the oriented graph with vertex-set all possible $(i,j,\varepsilon)$ and and an edge from $(i,j,\varepsilon)$ to $(i',j',\varepsilon')$ exactly when $i'=j$, one can reach any vertex from any vertex in finitely many steps --- actually at most 2.

Let $(i,j,\varepsilon)$ and let $(i',j',\varepsilon')$ be such that $i'=j$. By the second part of Claim~\ref{claim:connect}, we have vertices $u$ and $v$ in $\mathcal{H}_{i,j}^\varepsilon(c)$ such that $u_j<-3n\le \ell_{i',j'}^{\varepsilon'}\le 3n<v_j$ and a path connecting them in $\mathcal{H}_{i,j}^\varepsilon(c)$ that stays at $\|~\|_\infty$-distance at most $c$ from $\ell e_i+[-3n,3n]e_j$. By intermediate-value reasoning on the path, we can find vertices $x$ and $z$ on this path that are adjacent and satisfy $x_{i'}=x_j<\ell_{i',j'}^{\varepsilon'}\le z_j=z_{i'}$. Besides, we have $-3n\le x_{j'}\le 3n$ because either $j'=i$ and $x_{j'}=\ell$, or we have $j'\notin\{i,j\}$ and then comparison with $\ell e_i+[-3n,3n]e_j$ yields $|x_{j'}|\le c\le n$. Likewise, the inequality $-3n\le z_{j'}\le 3n$ holds. If we have $z_{i'}=\ell_{i',j'}^{\varepsilon'}$, then we are done, as $z$ belongs to $\mathcal{H}_{i',j'}^{\varepsilon'}(c)$ in this case. Otherwise, we have $x_{i'}<\ell_{i',j'}^{\varepsilon'}<z_{i'}$ and using Claim~\ref{claim:douanes} with $(i',j',\ell_{i',j'}^{\varepsilon'})$ and $(a,b)=(-3n,3n)$ concludes.
\end{proof}

\subsubsection{Proof of the existence of an infinite white cluster}
\label{sec:unique-setup}
By Proposition~\ref{prop:exist-white}, only the 2-dimensional case remains. By Claims~\ref{claim:douanes} and \ref{claim:connect}, one can simply mimic the staircase argument of Section~\ref{sec:2-dim}. The only adjustment to be made is that the long side of rectangles should not be defined by the values $(0,2^{n+1})$ but $(-2c,2^{n+1}+2c)$ instead.\hfill \qed

\subsubsection{Proof of Claim~\ref{claim:douanes}}
\label{sec:douanes}
Let us make the assumptions of the claim. Without loss of generality, we assume that $x_j<\ell<z_j$ rather than $z_j<\ell<x_j$ --- the proof of the other case is similar, or we can reduce it to the first one via minus the identity.

Let us first prove the result under the additional assumption $\ell=0$. In this case, we can use the second assumption of Theorem~\ref{thm:setup} to find some $y$ that is a neighbour of either $x$ or $z$ and such that $y_i=0$. Therefore, in $\mathsf{Cayley}(\Gamma,S)$, the vertex $y$ lies at distance at most 2 from $x$. As $a\le x_j\le b$, we have
\[
a-c\le a-2\max_{s\in S}\|s\|_\infty \le y_j\le b+2\max_{s\in S}\|s\|_\infty \le b+c.
\]
Combined with $y_i=0$, this gives $y\in H_{i,j}^\ell(a-c,b+c)$.

If $v\in \Gamma$, by translating everything by $v$, the case $\ell=v_i$ follows from the case $\ell=0$. As we assumed that, for all $i$, the map $\Gamma\ni x\longmapsto x_i\in\mathbb{Z}$ is onto, the claim holds for all values of $\ell$.\hfill \qed

\begin{rema}
    \label{rem:st}
    With the same argument, if we were to make the second assumption of Theorem~\ref{thm:setupblack} instead, we would be able to prove a variation of Claim~\ref{claim:douanes} with the relaxed assumption ``$x_i\le\ell\le z_i$ or $z_i\le \ell\le x_i$'' and with conclusion ``at least one of the vertices $x$ and $z$ must belong to $H_{i,j}^\ell(a-c,b+c)$''. Besides, taking $c\ge \max_{s\in S}\|s\|_\infty$ would suffice.
\end{rema}

\subsubsection{Proof of Claim~\ref{claim:connect}}
\label{sec:connect} Let $N$ be the index of $\Gamma$ in $\mathbb{Z}^d$. Let $(e_1,\dots,e_d)$ be the canonical basis of $\mathbb{Z}^d$. For every $i\in\llbracket 1,d\rrbracket$, there must be two distinct values $a_i,b_i\in\llbracket 0,N\rrbracket$ such that $a_i e_i$ and $b_i e_i$ belong to the same coset of $\Gamma$ in $\mathbb{Z}^d$. By taking the difference of two such values, we get some $c_i\in \llbracket 1,N\rrbracket$ such that $c_i e_i$ belongs to $\Gamma$. Therefore, $\Gamma$ contains $\bigoplus_i c_i\mathbb{Z}e_i$.

Let $(i,j,\ell,a,b)$ be as in the claim. Proceeding as in Claim~\ref{claim:douanes}, without loss of generality, we may and will assume that $\ell=0$.

Let us start by proving the first conclusion. Let $H_i:=\{x\in \Gamma\,:\,x_i=0\}$.
By the first assumption of Theorem~\ref{thm:setup}, for every $x\in H_i$, for every $c$ large enough, there is a path from the origin to $x$ that stays in $H_i\cap \llbracket -c,c\rrbracket^d$. Therefore, for every $c$ large enough, every element of the finite set $H_i\cap \prod_j \llbracket -c_j,c_j\rrbracket$ can be reached from the origin by a path that stays in $H_i\cap \llbracket -c,c\rrbracket^d$. Pick such a $c$. Observe that, by reversing paths, every element $x$ of the finite set $H_i\cap \prod_j \llbracket -c_j,c_j\rrbracket$ can be reached from the origin by a path that stays in $H_i\cap \llbracket x_k-c,x_k+c\rrbracket^d$.

Let $x$ and $y$ be two elements of $H_{i,j}^0(a-c,b+c)$. By our choice of $c$, the vertex $x$ (resp.~$y$) can be connected within $H_{i,j}^0(a-2c,b+2c)$ to some $x'$ (resp. $y'$) that belongs to $H_{i,j}^0(a-c,b+c)\cap \bigoplus_k c_k\mathbb{Z}e_k$. It thus suffices to connect $x'$ to $y'$ in $H_{i,j}^0(a-2c,b+2c)$. To do so, pick a geodesic path from $(\tfrac{x'_1}{c_1},\dots,\frac{x'_d}{c_d})$ to $(\tfrac{y'_1}{c_1},\dots,\frac{y'_d}{c_d})$ in the usual hypercubic lattice and then dilate it by $c_i$ in each direction. This gives a finite sequence of elements of $H_{i,j}^0(a-c,b+c)$ that starts at $x'$ and ends at $y'$. By our choice of $c$, each vertex $z$ of the sequence can be connected to the next one $z'$ by remaining in $H_i$ and by staying at $\|~\|_\infty$-distance at most $c$ from both $z$ and $z'$. Therefore, by concatenation, there is a path from $x$ to $z$ that stays in $H_{i,j}^0(a-2c,b+2c)$.

The proof of the second conclusion is similar. Pick $u=\lfloor \tfrac{a}{c_j}\rfloor c_j e_j$ and $v=\lceil \tfrac{b}{c_j}\rceil c_j e_j$. As $c\ge c_j$, the vertices $u$ and $v$ belong to $H_{i,j}^0(a-c,b+c)\cap c_j\mathbb{Z}e_j$. Consider the finite sequence of vertices leading from $u$ to $v$ by incrementing each time the position by $c_je_j$. Treating this sequence as in the previous paragraph produces a path that satisfies the desired properties. \hfill\qed

\subsection{Proof of Theorem~\ref{thm:setupblack}}
\label{sec:proba-proof-noire}
We want to proceed as in the proof of Section~\ref{sec:unique} except that we need to strengthen Item~\ref{item:adja} to ``any path from $\llbracket -m,m\rrbracket^d$ to infinity must intersect $\mathcal{H}$''. Under the assumptions of Theorem~\ref{thm:setupblack}, this is indeed possible, by Remark~\ref{rem:st}. Existence of the infinite white cluster is obtained as in Section~\ref{sec:unique-setup}.\hfill \qed

\subsection{The lattices $D_d$, $E_8$, and the Leech lattice fit our setup}
\label{sec:setup-proof}
In this section, we prove that $D_d$ satisfies the assumptions of Theorem~\ref{thm:setupblack} for $d\ge 3$ and that $E_8$ and the Leech lattice satisfy the assumptions of Theorem~\ref{thm:setup}. When we work with a lattice $\Gamma$ in $\mathbb{R}^d$, given $i\in\llbracket 1,d\rrbracket$, we write $H_i:=\{x\in\Gamma\,:\,x_i=0\}$.

\subsubsection{The lattice $D_d$ satisfies the assumptions of Theorem~\ref{thm:setupblack} for $d\ge3$.} 
\label{sec:dede}
Let us prove that the first hypothesis holds. By symmetry, let us assume without loss of generality that $i=d$. Let $x\in H_d$ and let us prove that there is a path connecting $x$ to 0 in $H_d$. By using the generators $\pm e_1\pm e_2\in S \cap  H_d$, we can stay in $H_d$ and go from $x$ to some vertex $y$ such that $y_1=0$. Similarly, by using $\pm e_j\pm e_{j+1}\in S \cap H_d$ for $j\le d-2$, we can find a vertex $z$ that can be reached from $x$ inside $H_d$ and that satisfies $z_1=\dots=z_{d-2}=0$.
As $z$ belongs to $H_d$, we also have $z_d=0$, and the last thing to do is to be able to adjust the $(d-1)^\text{th}$ entry. Since $z$ belongs to $H_d$, it must be the case that $z_{d-1}$ is even, by definition of $D_d$. Therefore, we would be done if $2e_{d-1}$ was an element of $S$. This is not the case but it can be emulated via a path of length 2 that stays in $H_d$, as $2e_{d-1}=(e_{d-1}+e_1)+(e_{d-1}-e_1)$ --- this is where we use the fact that $d\ge 3$. The first hypothesis is thus established.
As for the second one, it is clearly met as $S\subset \llbracket -1,1\rrbracket^d$.

\subsubsection{The lattice $E_8$ satisfies the assumptions of Theorem~\ref{thm:setup}.}
Let us prove that the first hypothesis holds. By symmetry, let us assume without loss of generality that $i=8$. Let $x\in H_8$ and let us prove that there is a path connecting $x$ to 0 in $H_8$. By definition of $E_8$ and as $x_8=0$, all entries of $x$ are even. Proceeding as in Section~\ref{sec:dede} with $\pm 2e_j\pm 2e_{j+1}$ for $j<7$, we reduce to the case where $x_1=\dots=x_6=0$ and $x_8=0$. By definition of $E_8$, this entails that $x_7$ is a multiple of 4. We conclude by emulating the missing generator $4e_7$ as the sum of $2e_7+2e_1$ and $2e_7-2e_1$.

Let us move on to the second hypothesis. Let $x$ and $z$ be two adjacent elements of $E_8$ that satisfy $x_i<0<z_i$ for some given $i\in \llbracket1,8\rrbracket$. As $S\subset \llbracket-2,2\rrbracket^8$, it must be the case that $x_i=-1$ and $z_i=1$. Using a generator of the form $(\pm1,\dots,\pm 1)$, we can thus find a neighbour $y$ of $x$ that belongs to $H_i$, and we could do the same for $z$.

\subsubsection{The Leech lattice satisfies the assumptions of Theorem~\ref{thm:setup}.}
Let us prove that the first hypothesis holds. As $M_{24}$ acts transitively on $\llbracket1,24\rrbracket$, let us assume without loss of generality that $i=24$ --- this is not essential, we could keep an abstract $i$ all along. Let $x\in H_{24}$ and let us prove that there is a path connecting $x$ to 0 in $H_{24}$. By Item~\ref{item:bit-zero} in the definition of the Leech lattice and as $x_{24}=0$, all entries of $x$ are even. By using generators of the form $\pm 4e_j\pm 4e_{j+1}$ for $j$ in $\llbracket1,22\rrbracket$, we can further assume that, for every $j\in\llbracket1,22\rrbracket$, we have $x_j\in\{0,2\}$. Since we can emulate $8e_{23}$ as $(4e_{23}+4e_1)+(4e_{23}-4e_1)$ and because we know $x_{23}$ to be even, we can also assume that $x_{23}\in \{0,2,4,6\}$. By Item~\ref{item:bit-two} in the definition of the Leech lattice, this gives $x_{23}\in\{0,2\}$, so that all $x_j$ belong to $\{0,2\}$. By Item~\ref{item:bit-one} in the definition of the Leech lattice, $\tfrac{x}{2}$ is an element of the extended binary Golay code, if one identifies $\{0,1\}$ with $\mathbb{F}_2$. By property of the extended binary Golay code, the number of nonzero entries of $x$ belongs to $\{0,8,12,16,24\}$. Denote this number by $K$. As $x_{24}=0$, the value $24$ is forbidden for $K$.

If $K=0$, then $x=0$ and we are done. If $K=8$, then $x\in S$ and we are done as well. Let us deal with the case $K=16$. The complement of the support of $x$ is then an octad. One can thus find two other octads such that the three of them form a trio. In other words, the support of $x$ can be seen as the union of two disjoint octads, which necessarily avoid $24$. This completes the case $K=16$, as this provides a path of length 2 connecting $x$ to 0 inside $H_{24}$. We are left with the case $K=12$. Proceeding as for 16, it suffices to prove that every dodecad avoiding $24$ can be realised as the symmetric difference of two octads that avoid 24 as well.

Let $\Delta$ be a dodecad avoiding $24$. It is known that $\Delta$ is expressible as the symmetric difference of two octads. Let us fix such a pair of octads. If they avoid 24, then we are done. Otherwise, pick some $k$ that belongs to none of the two octads --- in particular not to $\Delta$. Let $\Delta'$ denote the complement of $\Delta$, which is also a dodecad. As the group $M_{12}(\Delta)$ acts transitively on $\Delta'$, we can find some $\sigma\in M_{12}(\Delta)$ such that $\sigma(k)=24$. Applying $\sigma$ to our pair of octads produces a pair of octads avoiding $24$ and the symmetric difference of which is $\Delta$.

Let us move on to the second hypothesis. Let $i\in \llbracket1,24\rrbracket$. By looking at the list of elements of $S$, the image of $S$ by $x\longmapsto x_i$ is $\hat{S}:=\{-4,-3,-2-1,0,1,2,3,4\}$. Indeed, there is an octad containing $i$, and both vectors $(0,\dots,0)$ and $(1,\dots,1)$ belong to the extended binary Golay code. Let $x$ and $z$ be two adjacent elements of the Leech lattice that satisfy $x_i<0<z_i$. We must thus have $x_i\in \{-3,-2,-1\}$. But as $1$, $2$ and $3$ belong to $\hat{S}$, we can find a neighbour $y$ of $x$ such that $y_i=0$. All proofs are now complete.\hfill\qed

\bibliographystyle{smfalpha}
\bibliography{biblio}

\appendix
\section{A brief introduction to the Furstenberg--Zimmer Theorem}
\label{appen}

Let $\mathbb{Z}\curvearrowright (E,\mathscr{E},\nu)$ be a measurable measure-preserving action of $\mathbb{Z}$ on a probability space with a countably generated $\sigma$-field $\mathscr{E}$. We capture this measure-preserving dynamical system by the bijection $T:E\to E$ corresponding to the action of $1\in\mathbb{Z}$. This system is said to be \emph{weakly mixing} if, for every complex-valued $f\in \mathrm{L}^2(E,\nu)$, we have
\[
\tfrac{1}{n}\sum_{i=1}^n\left|\langle T^i f,f\rangle-\left({\textstyle \int f\,\mathrm{d}\nu}\right)^2\right|\xrightarrow[n\to\infty]{} 0.
\]
The system is \emph{compact} if for every complex-valued $f\in \mathrm{L}^2(E,\nu)$, the closure of the orbit $\{T^n f\,:\,n\in\mathbb{Z}\}$ in $\mathrm{L}^2$ is compact. A good way to digest these definitions is to take $f$ to be an indicator function.

\begin{enonce}{Exercise}[When each point has a period]
Assume that for almost every $x\in E$, there is some $n\ge 1$ such that $x$ is fixed by the action of $n\mathbb{Z}$. Prove that the system under consideration is compact.
\end{enonce}

\begin{enonce}{Exercise}[Kronecker systems]
\label{exo}
Let $K$ be a topological group that is abelian, metrisable and compact. Endow it with its Borel $\sigma$-algebra and its probability Haar measure. Let $\mathbb{Z}$ act by translations on $K$, i.e.~let $T$ be of the form $k\longmapsto k+a$ for some $a$.
Prove that the system under consideration is compact.
\end{enonce}

Recall that a dynamical system is \emph{ergodic} if and only if all $\mathbb{Z}$-invariant measurable subsets of $E$ have probability 0 or 1. In a way reminiscent of ``decomposing a group action into orbits'', every dynamical system can be ``decomposed into ergodic subsystems'', and ergodic systems cannot be further decomposed in this sense. This leaves us with the task of understanding ergodic systems. As both weakly mixing systems and compact systems are rather well understood, it would be satisfactory to be able to reduce arbitrary ergodic systems to situations that are either weakly mixing or compact.

Without getting through all definitions, the Furstenberg--Zimmer Theorem states that ``every ergodic system can be decomposed as a {\footnotesize(possibly trivial)} transfinite tower of extensions over the trivial one-point system such that each extension is compact conditioned on what is below it, except for the {\footnotesize(possibly trivial)} final extension which is weakly mixing conditioned on what is below it''. Once we have pumped out all order, what remains is disorder.

The reader may be puzzled to read the statement of a decomposition theorem of ergodic systems, knowing that ``ergodic'' means ``indecomposable in some sense''. All lies in the ``in some sense''. Ergodicity is an indecomposability relative to disjoint unions, to injections. The Furstenberg--Zimmer Theorem is a decomposition in terms of factors, a notion that revolves around surjections, projections, products. 
Let us go back to our analogy with group actions and orbits, where ``ergodic'' was analog to an action having exactly one orbit: an action with a single orbit may very well have interesting ``projective'' decompositions. For instance, let $\mathbb{Z}/4\mathbb{Z}$ act naturally on its 2-binary-digits representation $\{0,1\}^2$. This action indeed has a single orbit but it can be understood ``in layers'', by revealing first the digit of weight $2^0$ {\footnotesize(which corresponds to an action of $\mathbb{Z}/4\mathbb{Z}$ on $\{0,1\}$)} and only later that of weight $2^1$ {\footnotesize(revealing bits in the other order would not work: why?)}.

To avoid technicalities, we have focused on $\mathbb{Z}$. Such a theory can be developed for $\mathbb{Z}^d$ as well. Additional care is needed to state a clean decomposition theorem in this context, as some extensions of the decomposition may need to be ``compact in some directions and weakly mixing in others''.

\section{Hands-on explanations regarding the almost-periodic content of Bernoulli line percolation}
\label{appen:other}

Let us recall the definition of the so-called Bernoulli line percolation \cite{marcelo-vladas}. For simplicity, we pick $p\in[0,1]$ and work with $\mathbb{Z}^2$ endowed with its usual square grid structure. The process is defined as follows: independently, for each vertical and each horizontal line, erase all of its vertices with probability $1-p$. A vertex is removed if it is erased at least once and retained otherwise.

This process has both weakly-mixing and almost-periodic content. Let us explain this by focusing on one step of the construction, namely the definition of the ``erased due to horizontal lines'' status. If one contracts each horizontal line to a point, this is a collection of independent Bernoulli random variables of parameter $p$, the archetype of weakly mixing systems. We could say that this step of the construction is weakly mixing in the transverse direction, namely with respect to vertical translations. On the other hand, every erasure landscape defined by this step is, by definition, invariant under horizontal translations {\footnotesize(we insist that we speak of invariance of almost every individual configuration, not of the probability distribution)}. We could thus say that it is periodic with period given by the vector $(1,0)$. This observation accounts for the fact that the Furstenberg--Zimmer decomposition of Bernoulli line percolation has nontrivial almost-periodic content. Very similar explanations apply to all other percolation processes of the list \cite{jonasson, hoffman, pete, brochette, marcelo-vladas, ksv, hsst}.
\end{document}